\documentclass[a4paper,11pt]{article}
\usepackage{amsfonts}
\usepackage{amssymb}
\usepackage{amsmath}
\usepackage{amsthm}

\usepackage{cite}

\usepackage{comment}
\usepackage{hyperref}

\makeatletter
\newcommand{\leqnomode}{\tagsleft@true\let\veqno\@@leqno}
\newcommand{\reqnomode}{\tagsleft@false\let\veqno\@@eqno}
\makeatother

\usepackage[shortlabels]{enumitem}
\setlist[enumerate]{leftmargin=*,align=left,labelindent=\parindent}

\makeatletter
\newcommand{\myitem}[1][]{%
\item[#1]\protected@edef\@currentlabel{#1}\ignorespaces%
}
\makeatother

\newtheorem{theorem}{Theorem}[section]
\newtheorem{proposition}[theorem]{Proposition}
\newtheorem{lemma}[theorem]{Lemma}
\newtheorem{corollary}[theorem]{Corollary}

\theoremstyle{definition}
\newtheorem{definition}[theorem]{Definition}
\theoremstyle{remark}
\newtheorem{remark}[theorem]{Remark}
\newtheorem*{notation*}{Notation}
\newtheorem{notation}{Notation}
\newtheorem{question}[theorem]{Question}

\newcommand{\ha}{\mathsf{HA}}
\newcommand{\pa}{\mathsf{PA}}
\newcommand{\T}{{T}}

\newcommand{\DNE}[1]{{#1}\text{-}\mathrm{DNE}}
\newcommand{\LEM}[1]{{#1}\text{-}\mathrm{LEM}}
\newcommand{\DML}[1]{{#1}\text{-}\mathrm{DML}}
\newcommand{\DNS}[1]{{#1}\text{-}\mathrm{DNS}}
\newcommand{\NN}[1]{\neg \neg {#1}}

\newcommand{\PNFT}[2]{\mathrm{PNFT}_{T} \left( {#1}, {#2} \right) }

\newcommand{\F}{\mathrm{F}}
\newcommand{\U}{\mathrm{U}}
\newcommand{\E}{\mathrm{E}}
\newcommand{\FV}[1]{\mathrm{ FV} \left({#1}\right)}
\newcommand{\dn}[1]{{#1}^{\mathrm{dn}}}
\newcommand{\n}[1]{{#1}^{\mathrm{n}}}
\newcommand{\df}[1]{{#1}^{\mathrm{df}}}
\newcommand{\SN}[1]{{#1}_{N}}
\newcommand{\SNP}[1]{{\left( #1 \right)}_{N}}

\newcommand{\PH}{*}

\newcommand{\CTz}{\mathrm{CT_0}}

\newcommand{\ol}[1]{\overline{#1}}

\newcommand\ang[1]{{\langle #1 \rangle} }

\newcommand{\lr}{\leftrightarrow}

\newcommand{\llr}{\longleftrightarrow}

\newcommand{\lra}{\longrightarrow}
\newcommand{\vp}{\varphi}

\newcommand{\QF}[1]{{#1}_{\mathrm{qf}}}

\newcommand{\degree}[1]{\mathit{deg}(#1)}

\newcommand{\alt}[1]{\mathit{Alt}(#1)}
\DeclareSymbolFont{largesymbol}{OMX}{yhex}{m}{n}
\DeclareMathAccent{\Widehat}{\mathord}{largesymbol}{"62}

\title{Prenex normal form theorems in semi-classical arithmetic}

\author{Makoto Fujiwara\footnote{Email: makotofujiwara@meiji.ac.jp}
\footnote{School of Science and Technology, Meiji University, 1-1-1 Higashi-Mita, Tama-ku, Kawasaki-shi, Kanagawa 214-8571, Japan.}
and Taishi Kurahashi\footnote{Email: kurahashi@people.kobe-u.ac.jp}
\footnote{Graduate School of System Informatics,
Kobe University,
1-1 Rokkodai, Nada, Kobe 657-8501, Japan.}}
\date{\today}

\begin{document}
\maketitle

\begin{abstract}
%We revisit the prenex normal form theorem over semi-classical arithmetic developed in Akama et al. \cite{ABHK04}.
%We first provide a simple counterexample of their prenex normal form theorem, then modify it in an appropriate way.
%In addition, we show that our prenex normal form theorem is optimal from several perspectives.
Akama et al. \cite{ABHK04} systematically studied an arithmetical hierarchy of the law of excluded middle and related principles in the context of first-order arithmetic.
In that paper, they first provide a prenex normal form theorem
%for semi-classical arithmetical theories 
as a justification of their semi-classical principles restricted to prenex formulas.
%which they studied.
However, there are some errors in their proof.
In this paper, we provide a simple counterexample of their prenex normal form theorem \cite[Theorem 2.7]{ABHK04}, then modify it in an appropriate way.
In addition, we characterize several prenex normal form theorems with respect to semi-classical arithmetic.
%In particular, it follows from the investigation that our prenex normal form theorem is optimal.
%from several perspectives.  
%show that our prenex normal form theorem is optimal from several perspectives.
 \bigskip

%  \noindent\textsl{Keywords:} \TBR
  %Bar induction; Bar recursion; Continuity principle; Fan theorem; Intuitionistic mathematics; Constructive reverse mathematics\\[3pt]
 % \noindent\textsl{MSC2010:} \TBR
  %03F55; 03F35; 03F10; 03B30; 03B20
\end{abstract}

\section{Introduction}
\label{sec: Introduction} 
Prenex normal form theorem is one of the most basic theorems on theories based on classical first-order predicate logic.
In contrast, it does not hold for intuitionistic theories in general.
Therefore it does not make sense to consider an arithmetical hierarchy in an intuitionistic theory.
On the other hand, if one reasons in some semi-classical arithmetic which lies in-between classical arithmetic and  intuitionistic arithmetic, one can take an equivalent formula of the prenex normal form for any formula with low complexity.
%arithmetical complexity in the sense of classical arithmetic.
%In the context of intuitionistic first-order arithmetic, 
Akama et al.\cite{ABHK04} introduces the classes of formulas $\E_k$ and $\U_k$ which corresponds to the classes of classical $\Sigma_k$ and $\Pi_k$ formulas respectively, and showed that the former is equivalent to the class of formulas of $\Sigma_k$ form and the latter is so for $\Pi_k$ over some semi-classical arithmetic respectively.
This prenex normal form theorem justifies their investigation on the arithmetical hierarchy in the context of intuitionistic first-order arithmetic.
%such semi-classical arithmetic.
%Since formulas of the prenex normal form are usuful in syntactical investigations
Unfortunately, however, there are some crucial errors in their proof of the prenex normal form theorem \cite[Theorem 2.7]{ABHK04}.
In this paper, we revisit their formulation and modify their prenex normal form theorem in an appropriate way.

In Section \ref{sec: Preparation}, we recall the definitions and basic properties.
In Section \ref{sec: CE}, we provide a simple counterexample of \cite[Theorem 2.7]{ABHK04}.
In Section \ref{sec: PNFT}, we show the corrected version of the prenex normal form theorem (see Theorem \ref{thm: PNFT}).
In addition, we also present the simplified version of the prenex normal form theorem for formulas which do not contain the disjunction (see Theorem \ref{thm: PNFT for df-formulas}).
In Section \ref{sec: conservativity}, we carry out some generalization of a well-known result that classical arithmetic is $\Pi_2$-conservative over intuitionistic arithmetic with respect to semi-classical arithmetic.
In Section \ref{sec: Optimality}, using the generalized conservation result in Section \ref{sec: conservativity}, we characterize several prenex normal form theorems with respect to semi-classical arithmetic.
In particular, among other things, we show that for any theory $\T$ in-between intuitionistic arithmetic and classical arithmetic, $\T$ proves a semi-classical principle $\DNE{(\Pi_k \lor \Pi_k)}$ if and only if $\T$ satisfies the prenex normal form theorem for $ \U_{k'}$ and $\Pi_{k'}$ for all $k' \leq k$ (see Theorem \ref{thm: characterization of PNFT(Ukp,Pk)}).

%show that our prenex normal form theorem
%(Theorem \ref{thm: PNFT}) 
%is optimal from several perspectives.

 %\begin{notation}
%Without otherwise stated, the implications $\to$ or $\leftrightarrow$ in our (informally written) proofs denote derivations which are available in intuitionistic arithmetic.
%On the other hand, w
Throughout this paper, we work basically over intuitionistic arithmetic.
When we use some principle (including induction hypothesis [I.H.]) which is not available in intuitionistic arithmetic, it will be exhibited explicitly.
%we exhibit a logical principle (which is not provable in intuitionistic arithmetic) or induction hypothesis (I.H.)
%with the implications $\to$ or $\leftrightarrow$ explicitly
%when it is used in the derivation.
%indicate where additional logical principles or induction hypothesis (I.H.) are used in the derivations.
%\end{notation}
As regards basic reasoning over intuitionistic first-order logic, we refer the reader to see \cite[Section 6.2]{vD13}.

\section{Preparation}\label{sec: Preparation}
%In the sections except Section \ref{sec: PNFT without or}, 
Throughout this paper, we work with a standard formulation of intuitionistic arithmetic $\ha$ described e.g.~in \cite[Section 1.3]{Tro73}, which has function symbols for all primitive recursive functions.
%in the language.
%The classical variant $\pa$ is defined by adding the axiom scheme of the law of excluded middle to the axiom schemata of $\ha$.
We work in the language containing all the logical constants $\forall, \exists, \to, \land, \lor, \perp$.
%In Section \ref{sec: PNFT without or}, however, we also study in the language containing without the disjunction symbol $\lor$.
Let $\T$ denote a theory (e.g. $\ha$), and ${\rm P}$ and ${\rm Q}$ denote schemata (e.g. logical principles).
Then $\T + {\rm P}$ denotes the theory obtained from $\T$ by adding ${\rm P}$ into the axioms.
In particular, the classical variant $\pa$ is defined as $\ha + {\rm LEM}$, where ${\rm LEM}$ is the axiom scheme of the law of excluded middle.
We write $\T \vdash {\rm Q}$ (or $\T $ proves ${\rm Q}$) if any instance of ${\rm Q}$ is provable in $\T$.
 We write $\T \vdash {\rm P}+{\rm Q}$ if $\T \vdash {\rm P}$ and $\T \vdash {\rm Q}$.

\begin{notation}
For a formula $\vp$, $\FV{\vp}$ denotes the set of free variables in $\vp$.
Quantifier-free formulas are denoted with subscript ``qf'' as $\QF{\vp}$.
In addition, a list of variables is denoted with an over-line as $\ol{x}$.
In particular, a list of quantifiers of the same kind is denoted as $\exists \ol{x}$ and  $\forall \ol{x}$ respectively.
\end{notation}

\begin{definition}
\label{def: Sigma_k and Pi_k}
\noindent
The classes $\Sigma_k$ and $\Pi_k$ of formulas are defined as follows:
\begin{itemize}
\item
$\Sigma_0$, as well as $\Pi_0$, is the class of all quantifier-free formulas;
\item
$\Pi_{k+1}$ is the class of all formulas of form
$Q_1 \ol{x_1} \cdots Q_{k+1} \ol{x_{k+1}} \, \QF{\vp}$;
%(\ol{x_1}, \dots, \ol{x_{k+1}})$;
    \item 
$\Sigma_{k+1}$ is the class of all formulas of form
$Q'_1 \ol{x_1} \cdots Q'_{k+1} \ol{x_{k+1}} \, \QF{\vp}$;
%(\ol{x_1}, \dots, \ol{x_{k+1}})$;
\end{itemize}
where
$Q_i$ represents $\forall$ for odd $i$ and $\exists$ for even $i$
and
$Q'_i$ represents $\exists$ for odd $i$ and $\forall$ for even $i$.
Following \cite{ABHK04}, we define the classes $\Sigma_k$ and $\Pi_k$ in the non-cumulative manner (namely, each $Q_i \ol{x_i} $ and $Q'_i \ol{x_i}$ must not be empty).
A formula $\vp$ is of {\bf prenex normal form} if $\vp \in \Sigma_k \cup \Pi_k$ for some $k$. 
\end{definition}
\noindent
%Note that 

\begin{remark}
\label{rem: sigma_k and pi_k}
%Without loss of generality
Since the list of variables can be contracted into one variable in $\ha$ by using a fixed primitive recursive pairing function (see e.g. \cite[1.3.9]{Tro73}), one may assume that for each natural number $k>0$, a formula in $\Sigma_k$ is of form $\exists x \vp(x)$ with some $\vp(x) \in \Pi_{k-1}$ and a formula in $\Pi_k$ is of form $\forall x \psi(x)$ with some $\psi(x) \in \Sigma_{k-1}$ without loss of generality.
\end{remark}

\begin{lemma}
\label{lem: sigma_k and pi_k}
Let $k$ be a natural number.
Let $\vp$ be in $\Pi_k$ and $\psi$ be in $\Sigma_k$.
Then, for all natural numbers $i,j$, there exist $\vp', \psi'\in \Pi_{k+i}$ and $\vp'', \psi'' \in \Sigma_{k+j}$ such that $\FV{\vp}=\FV{\vp'}=\FV{\vp''}$, $\FV{\psi}=\FV{\psi'}=\FV{\psi''}$,
$
\ha \vdash  \vp \lr \vp' \lr \vp''   $
and $\ha \vdash \psi \lr \psi' \lr \psi''$.
\end{lemma}
\begin{proof}
Straightforward by the fact that
\begin{equation*}
\label{eq: adding quantifier}
\ha \vdash \xi \lr \forall z \xi  \lr \exists z \xi
\end{equation*}
for any $ z\notin \FV{\xi}$.
\end{proof}

\begin{definition}
For a class $\Gamma$ of formulas,
$\Gamma (\ol{x})$ denotes the class of formulas $\vp$ in $\Gamma$ such that $\FV{\vp} \subseteq \{ \ol{x}\}$.
\end{definition}

\begin{remark}
\label{rem: identification on Sigma_k and Pi_k}
In the light of Lemma \ref{lem: sigma_k and pi_k}, throughout this paper, we identify the classes $\Sigma_k$ and $\Pi_k$ with the classes defined as in Definition \ref{def: Sigma_k and Pi_k} with allowing the quantifiers $Q_i$ and $Q'_i$ to be empty.
Under this identification, for all $k$ and $k'$ such that $k< k' $, $\Pi_k (\ol{x}) $ and $\Sigma_k (\ol{x})$ are considered to be sub-classes of $ \Sigma_{k'} (\ol{x}) \cap \Pi_{k'} (\ol{x})$.
We frequently use this property in what follows.
\end{remark}

%In the following, w
Recall the logical principles from \cite{ABHK04} and related principles:
 \begin{definition}\label{def: LP}
Let $\Gamma$ and $\Gamma'$ be classes of formulas.
%Let $\Gamma (\ol{x})$ be a class of formulas $\vp$ in $\Gamma$ such that $\FV{\vp} \subseteq \{ \ol{x}\}$.
%which contains free variable  
\noindent
\begin{itemize}
    \item 
$\LEM{\Gamma}:\,  \forall x \left( \vp (x) \lor \neg \vp (x) \right)$ where $\vp (x)
\in \Gamma(x)$.
\item
$\DML{\Gamma}:\,  \forall x \left( \neg (\vp (x) \land \psi (x)) \to \neg \vp (x) \lor \neg \psi (x) \right)$ where $\vp (x), \psi (x) \in \Gamma(x)$.
    \item 
$\DNE{\Gamma}:\,  \forall x \left( \neg \neg \vp (x) \to \vp (x) \right) $ where $\vp (x) \in \Gamma(x)$.
    \item 
$\DNE{(\Gamma \lor \Gamma')}:\,  \forall x \left( \neg \neg (\vp (x) \lor \psi (x)) \to \vp (x) \lor \psi (x) \right)$ where $\vp (x)  \in \Gamma(x)$ and $\psi(x) \in \Gamma'(x)$.
\item
$\DNS{\Gamma}: \, \forall x \left(\forall  y \neg \neg \vp(x,y ) \to \neg \neg \forall y \vp(x,y) \right)$ where $\vp(x,y ) \in \Gamma (x,y)$.
\item
Let ${\rm P} \in \{\LEM{\Gamma}, \DML{\Gamma}, \DNE{\Gamma}, \DNE{(\Gamma \lor \Gamma')} , \DNS{\Gamma}\}$.\\          
$\NN{{\rm P}}:\,  \neg \neg \xi$ where $\xi$ is an instance of ${\rm P} $.
%which may contain $x$ free.
%\item
\end{itemize}
 \end{definition}
 \noindent
 Note that our logical principles are equivalent also to those defined with lists of quantifiers of the same kind (cf. Remark \ref{rem: sigma_k and pi_k}).
 
\begin{remark}\label{rem: On formulations}
One has to care about the formulation of the double negated variants.
That is, one has to take the double negations of the universal closure of the original logical principles as in Definition \ref{def: LP}.
The double negated variants defined as such are not provable in $\ha$, which has been overlooked in the proof of \cite[Theorem 2.7]{ABHK04} (see also Section \ref{sec: CE}).
In fact, one may think of the double negated versions as variants of the double negation shift principle (see \cite{FK18}).
In addition, our double negated versions are equivalent to (the universal closures of) those with allowing free variables (cf. \cite[Remark 2.5]{FK18}).
\end{remark}

\begin{remark}
\label{rem: NNDNS <-> DNS}
For any class $\Gamma$ of formulas, 
$\DNS{\Gamma} $ is intuitionistically equivalent to $\NN{\DNS{\Gamma} }$ since
$$
\begin{array}{cl}
     & \forall x \left(  \forall y \neg \neg \vp  \to  \neg \neg   \forall y \vp  \right)  \\
 \llr    & \forall x \neg \neg \left(  \forall y \neg \neg \vp  \to   \neg \neg \forall y \vp  \right) \\
  \llr    & \neg \neg  \forall x \neg \neg \left(  \forall y \neg \neg \vp  \to   \neg \neg \forall y \vp  \right) \\
    \llr  &\neg \neg \forall x  \left(  \forall y \neg \neg \vp  \to  \neg \neg  \forall y \vp  \right).\\
\end{array}
$$
 \end{remark}
 
Next we reformulate the classes of formulas studied in \cite{ABHK04}.
The classes $\F_k, \U_k$ and $\E_k$ in Definition \ref{def: Classes} below were  dealt with in \cite{ABHK04} informally.
Here we shall introduce them and two additional classes $\U_k^+$ and $\E_k^+$ in a formal manner.
\begin{definition}\label{def: alp}
An alternation path is a finite sequence of $+$ and $-$ in which $+$ and $-$ appear alternatively.
For an alternation path $s$, let $i(s)$ denote the first symbol of $s$ if $s \not \equiv \ang{\, }$ (empty sequence); $ \times$ if $s \equiv \ang{\, }$.
%, 
Let $s^{\perp}$ denote an alternation path which is obtained by switching $+$ and $-$ in $s$, and let $l(s) $ denote the length of $s$.
\end{definition}

\begin{definition}\label{def: ALT(vp)}
For a formula $\vp$, the set of alternation paths $\alt{\vp}$ of $\vp$ is defined as follows:
\begin{itemize}
    \item 
    If $\vp$ is quantifier-free, then $\alt{\vp} := \{ \ang{\, } \}$;
    \item
    Otherwise, $\alt{\vp}$ is defined inductively by the following rule:
    \begin{itemize}
        \item 
        If $\vp \equiv \neg \vp_1$, then $\alt{\vp} := \{ s^{\perp} \mid s\in \alt{\vp_1}\}$;
        \item
        If $\vp \equiv \vp_1 \land \vp_2$ or $\vp \equiv \vp_1 \lor \vp_2$, then $\alt{\vp} := \alt{\vp_1} \cup \alt{\vp_2}$;
        \item
        If $\vp \equiv \vp_1 \to \vp_2$, then $\alt{\vp} := \{ s^{\perp} \mid s \in \alt{\vp_1}\} \cup \alt{\vp_2}$;
\item
If $\vp \equiv \forall x \vp_1 $, then $\alt{\vp} :=\{s \mid s\in \alt{\vp_1} \text{ and } i(s)\equiv -\} \cup \{-s \mid s\in \alt{\vp_1} \text{ and } i(s)\not \equiv - \} $;
\item
If $\vp \equiv \exists x \vp_1 $, then $\alt{\vp} :=\{s \mid s\in \alt{\vp_1} \text{ and } i(s)\equiv + \} \cup \{+s \mid s\in \alt{\vp_1} \text{ and } i(s)\not \equiv + \} $.
    \end{itemize}
\end{itemize}
In addition, for a formula $\vp$, the degree $\degree{\vp}$ of $\vp$ is defined as 
$$\degree{\vp} := \max \{l(s) \mid s \in \alt{\vp}  \} .$$
%\begin{itemize}
%    \item 
 %   If $\vp$ is quantifier-free, then $\degree{\vp} := 0$; 
 %   \item
%Otherwise, $\degree{\vp} := \max \{l(s) \mid s \in \alt{\vp}  \}$.
%\end{itemize}
\end{definition}

\begin{definition}
\label{def: Classes}
The classes $\F_k, \U_k, \E_k $ (from \cite[Definition 2.4]{ABHK04}), $\U_k^+ $ and $ \E_k^+ $ of formulas are defined as follows:
%in Definition 2.4 of Kurahashi-san's note (based on \cite[Definition 2.4]{ABHK04}):
\begin{itemize}
    \item
    $\F_k := \{ \vp \mid \degree{\vp}=k  \}    $;
    \item
    $\U_0:=\E_0:=\F_0$;
    \item
$\U_{k+1} := \{ \vp \in \F_{k+1} \mid i(s) \equiv - \text{ for all }s\in \alt{\vp} \text{ such that }l(s) =k+1 \}$;
\item
$\E_{k+1} := \{ \vp \in \F_{k+1} \mid i(s) \equiv + \text{ for all }s\in \alt{\vp} \text{ such that }l(s) =k+1  \}$;
%\item
%$\F_k^+ := \{ \vp \mid \degree{\vp}\leq k  \}    $;
    \item
$\begin{displaystyle}
\U_k^+ := \U_k \cup \bigcup_{i<k} \F_i
\end{displaystyle}
$;
    \item
$\begin{displaystyle}
\E_k^+ := \E_k \cup \bigcup_{i<k} \F_i
\end{displaystyle}
$. 
\end{itemize}
\end{definition}
\begin{remark}
\label{rem: Ekp and Ukp}
A similar property as Lemma \ref{lem: sigma_k and pi_k} also holds for $\U_k^+$ and $\E_k^+$:
for any $\vp \in \U_k^+$ and $\psi\in \E_k^+$, there exist $\vp'\in \U_k$ and $\psi' \in \E_k$ such that $\FV{\vp}=\FV{\vp'}$, $\FV{\psi}=\FV{\psi'}$, $\ha \vdash \vp \lr \vp'$ and $\ha \vdash \psi \lr \psi'$.
\end{remark}

Note $\F_0 =\Sigma_0 =\Pi_0$.
For each formula  $\vp \in \E_k$ (resp. $\psi \in \U_k$) of $\pa$,
%by prenex normal form theorem, 
one can take a formula $\vp' \in \Sigma_k$ (resp. $\psi' \in \Pi_k$) of $\pa$ which is equivalent to $\vp$ (resp. $\psi$) over $\pa$.
On the other hand, this is not the case for $\ha$.
In what follows, we study what kind of semi-classical arithmetic in-between $\pa$ and $\ha$ captures this property for each $k$.
%Note that the classes $\U_k^+$ and $\E_k^+$ are classically identical with the classes $\Pi_k$ and $\Sigma_k$.
%Our prenex normal form theorem (Theorem \ref{thm: PNFT}) below justifies this.
%\end{remark}
In fact, Akama et al. \cite{ABHK04} has already undertaken this.
%sort of work.
In particular, \cite[Theorem 2.7]{ABHK04} asserts the following:
\begin{enumerate}
    \item 
     For any $\varphi \in \E_k$,
        there exists $\varphi' \in \Sigma_k$  such that %$\FV{\vp}=\FV{\vp'}$ and
    $$
    \ha  + \DNE{\Sigma_k} \vdash \varphi \leftrightarrow \varphi'.
$$
\item
 For any $\varphi \in \U_k$,
        there exists $\varphi' \in \Pi_k$  such that %$\FV{\vp}=\FV{\vp'}$ and
    $$
    \ha  + \DNE{(\Pi_k \lor \Pi_k)} \vdash \varphi \leftrightarrow \varphi'.
$$
\end{enumerate}
However, the first assertion is wrong as we show in Section \ref{sec: CE}.

\section{A counter example}
\label{sec: CE}
 Recall that \cite[Theorem 2.7]{ABHK04} asserts that
 for any $\varphi \in \E_k$,
        there exists $\varphi' \in \Sigma_k$  such that %$\FV{\vp}=\FV{\vp'}$ and
    $$
    \ha  + \DNE{\Sigma_k} \vdash \varphi \leftrightarrow \varphi'.
$$
%As already mentioned in Section \ref{sec: Introduction}, 
However, there are some errors in the proof.
%of \cite[Theorem 2.7]{ABHK04}.
%the prenex normal form theorem 
In particular, in \cite[page 5, lines 15-17]{ABHK04}, it is written that
``
Since the double negations of ${\rm DNE}$ is intuitionistically provable, $\vdash_{\ha} \NN{A_0} \lr \NN{\exists x_0. C_0 }$ (which means $\ha \vdash  \NN{A_0} \lr \NN{\exists x_0 C_0 }$ in our notation)
''.
As studied in \cite{FK18}, however, the double negations of (the universal closure of) ${\rm DNE}$ is not provable in $\ha$, and hence, their proof actually uses some double negated logical principles in the sense of Definition \ref{def: LP}.
%formalized in an appropriate way (see Definition \ref{def: LP} and Remark \ref{rem: On formulations}) 
Our counterexample below
%(Proposition \ref{prop: counterexample} below)
shows that such a use of some additional principle is unavoidable.

Recall the arithmetical form of Church's thesis from \cite[3.2.14]{Tro73}:
$$
\CTz: \forall x \exists y \,\vp (x,y) \to \exists e \forall x \exists v \left({\rm T}(e,x,v) \land \vp (x, {\rm U}(v))  \right),
$$
where
${\rm T}$ and ${\rm U}$ are the standard primitive recursive predicate and function from the Kleene normal form theorem.
Note that $\CTz$ is a sort of combination of so-called  Church's thesis stating that every function is recursive and the countable choice principle (see \cite[4.3.2]{ConstMathI}).

\begin{proposition}
\label{prop: counterexample}
The following sentence
$$
\varphi_0:\equiv \neg \forall x \left( \neg  \exists u \left( {\rm T}(x,x,u) \land {\rm U}(u)=0 \right) \lor \neg  \exists u \left( {\rm T}(x,x,u) \land {\rm U}(u)\neq 0 \right) \right)
$$
%in $  \E_1$,
is not equivalent to any sentence $\varphi_0 ' \in \Sigma_1 $ over $\ha +\DNE{\Sigma_1}$.
\end{proposition}
\begin{proof}
We first claim that $\ha + \CTz$ proves $\varphi_0$.
For the sake of contradiction, assume 
\begin{equation}\label{eq: nonrec}
\forall x \left( \neg \exists u \left( {\rm T}(x,x,u) \land {\rm U}(u)=0 \right) \lor \neg \exists u \left( {\rm T}(x,x,u) \land {\rm U}(u)\neq 0 \right) \right)
\end{equation}
and reason in $\ha + \CTz$.
Since
$\vp_1\lor \vp_2 \lr \exists k\left( (k=0 \to \vp_1) \land (k\neq 0 \to \vp_2) \right)$ (see \cite[1.3.7]{Tro73}),
by $\CTz$, there exists $e$ such that
$$
\forall x \exists v
\left(
\begin{array}{l}
{\rm T}(e,x,v) \\
\land  \left({\rm U}(v)=0 \to \neg \exists u \left( {\rm T}(x,x,u) \land {\rm U}(u)=0 \right) \right) \\
\land  \left({\rm U}(v)\neq 0 \to \neg \exists u \left( {\rm T}(x,x,u) \land {\rm U}(u) \neq 0 \right) \right)
\end{array}
\right).
$$
In particular, for that $e$, there exists $v_e$ such that ${\rm T}(e,e,v_e)$,
$$
{\rm U}(v_e)=0 \to \neg \exists u \left( {\rm T}(e,e,u) \land {\rm U}(u)=0 \right)
$$
and 
$$
 {\rm U}(v_e)\neq 0 \to \neg \exists u \left( {\rm T}(e,e,u) \land {\rm U}(u) \neq 0 \right).
 $$
Since  ${\rm U}(v_e)=0 \lor {\rm U}(v_e) \neq 0$, we obtain a contradiction straightforwardly.

%Since $\varphi_0 $ is in $\E_1^+$, 

If $\varphi_0 $ is equivalent to some sentence $\varphi_0 ' \in \Sigma_1 $ over $\ha +\DNE{\Sigma_1}$, we have
$\ha + \DNE{\Sigma_1} +\CTz \vdash \varphi_0 '$
from the above claim.
%Since $\varphi_0 ' $ is almost negative (see \cite[3.2.9]{Tro73}),
Since $\varphi_0' \in \Sigma_1$, by the soundness of Kleene realizability  (see \cite[3.2.22]{Tro73}), we have that
$$\ha + \DNE{\Sigma_1}  \vdash \varphi_0 ',$$
and hence, $\ha + \DNE{\Sigma_1}  \vdash \varphi_0 $.
On the other hand, since
$$\forall x \left( 
\neg \left(
 \exists u \left( {\rm T}(x,x,u) \land {\rm U}(u)=0 \right) \land  \exists u \left( {\rm T}(x,x,u) \land {\rm U}(u)\neq 0 \right)
\right) \right)
$$
is provable in $\ha$, we have
%it is straightforward to see that
$\ha +\DML{\Sigma_1} \vdash \eqref{eq: nonrec}$.
Therefore we have
$$
\ha + \DNE{\Sigma_1} + \DML{\Sigma_1} \vdash \perp,$$
and hence, $\pa \vdash \perp$, which is a contradiction. 
\end{proof}

\begin{remark}
%On the other hand, o
One can easily see that $\vp_0 $ in Proposition \ref{prop: counterexample} is in $\E_1$.
Thus Proposition \ref{prop: counterexample} shows that $\vp_0 $ is a counterexample of \cite[Theorem 2.7]{ABHK04} for $k=1$.
\end{remark}

\section{Basic lemmata}
In this section, we show several lemmata which we use in the proofs of our prenex normal form theorems.
\begin{lemma}
\label{lem: HA+P |- vp => HA+NNP |- NNvp}
For any logical principle ${\rm P}$ in Definition \ref{def: LP} and any formula $\vp$ (possibly containing free variables), if $\ha + {\rm P} \vdash \vp$, then $\ha + {\rm \NN{P}} \vdash \NN{\vp}$.
\end{lemma}
\begin{proof}
Assume $\ha + {\rm P} \vdash \vp$.
Then there exists finite instances $\psi_1, \dots, \psi_k$ of ${\rm P}$ such that
$\ha + \psi_1 + \dots + \psi_k \vdash \vp$.
%If $\ha + {\rm P} \vdash \vp$, s
Since $\ha$ satisfies the deduction theorem, we have that $\ha$ proves $\psi_1 \land \dots \land \psi_k \to \vp$, and hence, $\NN{\left(\psi_1 \land \dots \land \psi_k \to \vp\right)}$, which is equivalent to $\NN{\psi_1} \land \dots \land \NN{\psi_k} \to \NN{\vp}$.
%, which is equivalent to $\NN{\psi_k} \to \NN{\vp}$.
Then we have $\ha + {\rm \NN{P}} \vdash \NN{\vp}$.
\end{proof}

\begin{corollary}
\label{cor: HA+P |- vp_1 <-> vp_2 => HA+NNP |- NNvp_1 <-> NNvp_2}
For any logical principle ${\rm P}$ in Definition \ref{def: LP} and any formulas $\vp_1$ and $\vp_2$ (possibly containing free variables), if $\ha + {\rm P} \vdash \vp_1 \lr \vp_2$, then $\ha + {\rm \NN{P}} \vdash \NN{\vp_1} \lr \NN{\vp_2}$.
\end{corollary}
\begin{proof}
Immediate from Lemma \ref{lem: HA+P |- vp => HA+NNP |- NNvp} and the fact that $\NN{\left(\vp_1 \lr \vp_2\right)}$ is intuitionistically equivalent to $\NN{\vp_1} \lr \NN{\vp_2}$.
\end{proof}

\begin{lemma}
\label{lem: basic facts on arithmetical hierarchy} 
Let $k$ be a natural number.
Let $\vp_1$ and $\vp_2$ be formulas in $\Sigma_k$, and let $\vp_3$ and $\vp_4$ be formulas in $\Pi_k$.
Then the following hold:
\begin{enumerate}
    \item
        \label{item: Sigma land}
     There exists a formula $\vp \in \Sigma_k$ such that $\FV{\vp}= \FV{\vp_1}\cup \FV{\vp_2} $ and $\ha \vdash \vp \lr  \vp_1 \land \vp_2$;
% which is equivalent to $\vp_1 \land \vp_2$ over $\ha$;
    \item
    \label{item: Pi land}
     There exists a formula $\vp' \in \Pi_k$ such that $\FV{\vp'}= \FV{\vp_3}\cup \FV{\vp_4} $ and $\ha \vdash \vp' \lr  \vp_3 \land \vp_4$.
%    There exists a formula $\vp \in \Pi_k$  which is equivalent to $\vp_3 \land \vp_4$ over $\ha$.
%        \item
 %there exists a $\Pi_k$-formula (resp. $\Sigma_k$-formula) $\vp$  which is equivalent to $\vp_1 \land \vp_2$ 
\end{enumerate}
\end{lemma}
\begin{proof}
%\eqref{item: Pi land} and \eqref{item: Sigma land} are shown
Straightforward by simultaneous induction on $k$.
%: Immediate by \cite[Lemma 6.2.1]{vD}.
%One can take the prenex normal form in $\ha.$
\end{proof}

\begin{lemma}
\label{lem: exists AvB}
For any formulas $\vp_1$ and $\vp_2$ in $\Sigma_k$,
%\item
 %\label{item: Sigma or}
 there exists a formula $\vp \in \Sigma_k $ such that $\FV{\vp}= \FV{\vp_1}\cup \FV{\vp_2} $ and $\ha \vdash \vp \lr  \vp_1 \lor \vp_2$.
\end{lemma}
\begin{proof}
%\eqref{item: Sigma or}:
Note that $\vp_1 \lor \vp_2$ is equivalent to
\begin{equation*}\label{eq: vp3 or vp4 without v}
\exists k \left(\left(k=0 \to \vp_1 \right) \land \left( k\neq 0 \to \vp_2 \right) \right)
\end{equation*}
over $\ha$ (see \cite[1.3.7]{Tro73}).
Since $\QF{\vp} \to \exists x \psi(x)$ and  $\QF{\vp} \to \forall x \psi(x)$ are equivalent to $\exists x\left(\QF{\vp} \to  \psi(x)\right)$ and  $\forall x \left( \QF{\vp} \to \psi(x) \right)$ respectively over $\ha$ when $x \notin \FV{\QF{\vp}}$, our assertion follows from Lemma \ref{lem: basic facts on arithmetical hierarchy} straightforwardly.
\end{proof}

\begin{lemma}
\label{lem: basic facts on our classes} 
Let $k$ be a natural number greater than $0$.
%$k$, the following hold:
\begin{enumerate}
    \item
    \label{item: EU and}
    If $\vp_1 \land \vp_2$ is in $\U_k^+$ (resp. $\E_k^+$) if and only if both of $\vp_1 $ and $\vp_2$ are in  $\U_k^+$ (resp. $\E_k^+$).
    \item
        \label{item: EU or}
    $\vp_1 \lor \vp_2$ is in $\U_k^+$ (resp. $\E_k^+$) if and only if both of $\vp_1 $ and $\vp_2$ are in  $\U_k^+$ (resp. $\E_k^+$).
    \item
            \label{item: EU to}
 $\vp_1 \to \vp_2$ is in $\U_k^+$ (resp. $\E_k^+$) if and only if $\vp_1 $ is in $\E_k^+$ (resp. $\U_k^+$) and $\vp_2$ is in  $\U_k^+$ (resp. $\E_k^+$).
 
 \item
         \label{item: U forall}
 $\forall x \vp_1$ is in $\U_k^+$ if and only if $\vp_1$ is in $\U_k^+$.
 \item
          \label{item: E exists}
 $\exists x \vp_1$ is in $\E_k^+$ if and only if  $\vp_1$ is in $\E_k^+$.
  \item
         \label{item: E forall}
$\forall x \vp_1$ is in $\E_{k+1}^+$ if and only if it is  in $\U_k^+$.
   \item
            \label{item: U exists}
$\exists x \vp_1$ is in $\U_{k+1}^+$ if and only if it is in $\E_k^+$.
   \begin{comment}
   \item
   \label{item: F forall}
   $\forall x \vp_1$ is in $\F_{k}^+$ if and only if it is in $\U_k^+$.
\item
   \label{item: F exists}
   $\exists x \vp_1$ is in $\F_{k}^+$ if and only if it is in $\E_k^+$.
\end{comment}
\end{enumerate}
\end{lemma}
\begin{proof}
\eqref{item: EU and}:
Assume $\vp_1 \land \vp_2 \in \U_k^+$.
Then $l(s) \leq k$ for all $s \in \alt{ \vp_1 \land \vp_2} = \alt{ \vp_1} \cup \alt{ \vp_2}$.
\begin{itemize}
\item
If $l(s) < k$ for all $s \in \alt{ \vp_1}$, then $\vp_1$ is in $\begin{displaystyle}
\bigcup_{i<k} \F_i  \subseteq \U_k^+
\end{displaystyle}
$.
\item
Otherwise, there is $s_0 \in \alt{ \vp_1}$ such that $l(s_0)=k$.
%we have $\degree{\vp_1}=k$.
Then, since $\vp_1 \land \vp_2 \notin \begin{displaystyle}
\bigcup_{i<k} \F_i 
\end{displaystyle}
$, we have $\vp_1 \land \vp_2 \in \U_k$.
Then, for each $s \in \alt{\vp_1}$ such that $l(s)=k$, we have $i(s)\equiv - $ since $s \in \alt{\vp_1 \land \vp_2 }$.
Thus $\vp_1 \in \U_k \subseteq \U_k^+$.
\end{itemize}
We also have $\vp_2 \in \U_k^+$ in the same manner.

For the converse direction, assume that $\vp_1 $ and $\vp_2$ are in $\U_k^+$.
Then, for all $s\in \alt {\vp_1 \land \vp_2}$, since $s\in \alt{\vp_1}$ or $s\in \alt{\vp_2}$, we have $l(s)\leq k$, in particular, $i(s) \equiv -$ if $l(s)=k$.
Thus $\vp_1 \land \vp_2 $ is in $\U_k^+$.

As for the case of $\E_k^+$, an analogous proof works.
%In the same manner, one can also show the case.
%Next assume $\vp_1 \land \vp_2 \in \E_k^+$.
%Then again $l(s) \leq k$ for all $s \in \alt{ \vp_1 \land \vp_2} = \alt{ \vp_1} \cup \alt{ \vp_2}$.

\noindent
\eqref{item: EU or}:
Analogous to \eqref{item: EU and}.

\noindent
\eqref{item: EU to}:
Assume $\vp_1 \to \vp_2 \in \U_k^+$.
Let $s$ be in $\alt{\vp_1 }$.
By the definition of $\alt{\vp_1 \to \vp_2}$, we have $s^{\perp} \in \alt{\vp_1 \to \vp_2}$ and $l(s) \leq k$.
\begin{itemize}
\item
If $l(s) < k$ for all $s \in \alt{ \vp_1}$, then $\vp_1$ is in $\begin{displaystyle}
\bigcup_{i<k} \F_i  \subseteq \E_k^+
\end{displaystyle}
$.
\item
Otherwise,  there is $s_0 \in \alt{ \vp_1}$ such that $l(s_0)=k$.
Since ${s_0}^{\perp} \in \alt{\vp_1 \to \vp_2}$, we have $\vp_1 \to \vp_2 \in \U_k$.
Then, for each $s \in \alt{\vp_1}$ such that $l(s)=k$, we have $i(s^{\perp})\equiv -$, and hence, $i(s)\equiv +$.
Thus $\vp_1 \in \E_k \subseteq \E_k^+$.
\end{itemize}
We also have $\vp_2 \in \U_k^+$ in the same manner.

For the converse direction, assume $\vp_1 \in \E_k^+$ and $\vp_2 \in \U_k^+$.
Since $\degree{\vp_1}\leq k$ and $\degree{\vp_2}\leq k$, we have $\degree{\vp_1 \to \vp_2}\leq k$.
\begin{itemize}
\item
If $\degree{\vp_1 \to \vp_2}<k$, then $\vp_1 \to \vp_2\in 
\begin{displaystyle}
\bigcup_{i<k} \F_i  \subseteq \U_k^+
\end{displaystyle}
$.
\item
If $\degree{\vp_1 \to \vp_2}=k$, for all $s\in \alt{\vp_1 \to \vp_2}$ such that $l(s)=k$, we have $s\in \alt{\vp_2}$ or $s\equiv {s_0}^{\perp}$ for some $s_0\in \alt{\vp_1}$.
In the former case, we have $i(s) \equiv -$ by $\vp_2 \in \U_k^+$.
In the latter case, we have  $i(s_0) \equiv +$ by $\vp_1 \in \E_k^+$, and hence, $i(s)\equiv -$.

\end{itemize}

%As for the case of $\E_k^+$, an 
One can also show that $\vp_1 \to \vp_2$ is in $\E_k^+$ if and only if $\vp_1 $ is in $\U_k^+$ and $\vp_2$ is in $\E_k^+$
analogously.

\noindent
\eqref{item: U forall}:
Assume $\forall x \vp_1 \in \U_k^+$.
\begin{itemize}
\item
If $\forall x \vp_1 \notin \U_k$, then $\forall x \vp_1 \in \begin{displaystyle}
\bigcup_{i<k} \F_i 
\end{displaystyle}$.
Since $\degree{\vp_1}\leq \degree{\forall x \vp_1} <k$, we have $\vp_1 \in \begin{displaystyle}
\bigcup_{i<k} \F_i \subseteq \U_k^+
\end{displaystyle}$.
\item
Otherwise, $\degree{\vp_1}\leq \degree{\forall x \vp_1} =k$.
If $\degree{\vp_1}< k$, then we have $\vp_1 \in \begin{displaystyle}
\bigcup_{i<k} \F_i \subseteq \U_k^+
\end{displaystyle}$.
Assume $\degree{\vp_1}= k$.
Let  $s$ be an alternation path of $\vp_1$ such that $l(s)=k$.
If $i(s) \not \equiv -$, by the definition of $\alt{\forall x \vp_1}$, we have $-s\in \alt{\forall x \vp_1}$, which contradicts $\degree{\forall x \vp_1} =k$ since $l(-s)=k+1$.
Then we have $i(s)\equiv -$.
Thus we have $\vp_1 \in \U_k \subseteq \U_k^+$.
\end{itemize}
For the converse direction, assume  $\vp_1  \in \U_k^+$.
\begin{itemize}
    \item 
If $\vp_1 \notin \U_k$, then $\vp_1 \in \begin{displaystyle}
\bigcup_{i<k} \F_i 
\end{displaystyle}$.
Thus $\degree{\vp_1}<k$, and hence, $\degree{\forall x \vp_1}\leq k$.
If $\degree{\forall x \vp_1}< k$, then $\forall x \vp_1  \in \begin{displaystyle}
\bigcup_{i<k} \F_i 
\end{displaystyle} \subseteq \U_k^+$.
If $\degree{\forall x \vp_1}= k$, since $i(s)\equiv -$ for all $s\in \alt{\forall x \vp_1}$, we have $\forall x \vp_1  \in \U_k  \subseteq \U_k^+$.
\item
Otherwise, $\degree{\vp_1}=k$ and $i(s)\equiv -$ for all $s \in \alt{\vp_1}$ such that $l(s) =k$.
By the definition of $\alt{\forall x\vp_1}$, for all $s \in  \alt{ \forall x \vp_1}$, we have $l(s)\leq k$, and hence, $\degree{\forall x\vp_1}= k$.
In addition, again by the definition of $\alt{\forall x\vp_1}$, we have $i(s)=-$ for all $s\in \alt{\forall x \vp_1}$ such that $l(s) =k$.
Thus $\forall x \vp_1  \in \U_k  \subseteq \U_k^+$.
\end{itemize}

\noindent
\eqref{item: E exists}:
Analogous to \eqref{item: U forall}.

\noindent
\eqref{item: E forall}:
Assume $\forall x \vp_1 \in \E_{k+1}^+$.
Since $i(s)\equiv -$ for all $s\in \alt{\forall x \vp_1}$, $\forall x \vp_1$ is not in $\E_{k+1}$.
Then $\forall x \vp_1 \in \begin{displaystyle}
\bigcup_{i\leq k} \F_i 
\end{displaystyle}$, and hence, $\degree{\forall x \vp_1} \leq k$.
\begin{itemize}
\item
If $\degree{\forall x \vp_1} < k$, then $\forall x \vp_1 \in \begin{displaystyle}
\bigcup_{i< k} \F_i  \subseteq \U_k^+
\end{displaystyle}
$.
\item
If $\degree{\forall x \vp_1} = k$, since $i(s) \equiv -$ for all $s\in \alt{\forall x \vp_1}$, we have $\forall x \vp_1 \in \U_k\subseteq \U_k^+$.
\end{itemize}
%Thus we have  $\forall x \vp_1 \in \U_k^+$.
%, and hence, $ \vp_1 \in \U_k^+$ by \eqref{item: U forall}.

%For the converse direction, assume $forall x \vp_1 \in \U_k^+$.
%By \eqref{item: U forall}, we have
The converse direction is trivial since $ \U_k^+ \subseteq
 \begin{displaystyle}
\bigcup_{i< k+1} \F_i  \subseteq \E_{k+1}^+
\end{displaystyle}
$.

\noindent
\eqref{item: U exists}:
Analogous to \eqref{item: E forall}.
\begin{comment}
\noindent
\eqref{item: F forall}:
Since $\U_k^+\subseteq \F_k^+$, it suffices to show the ``only if'' direction.
Assume $\forall x \vp_1 \in \F_{k}^+$.
\begin{itemize}
    \item 
If $\degree{\forall x \vp_1}< k$, we have $\forall x \vp_1 \in \F_{k-1}^+ \subseteq  \U_{k}^+$.
\item
If $\degree{\forall x \vp_1}=k$, since $i(s) \equiv -$ for all $s\in \alt{\forall x \vp_1}$, we have $\forall x \vp_1 \in \U_k\subseteq \U_k^+$.
\end{itemize}

\noindent
\eqref{item: F exists}:
Analogous to \eqref{item: F forall}.
\end{comment}
\end{proof}

\begin{lemma}
\label{lem: HA+NNS_{k-1}-DNE |- NN(Pk v Pk) <-> NNPk}
Let $k$ be a natural number.
For all $\vp_1$ and $\vp_2$ in $\Pi_k$, there exists $\vp \in \Pi_k$ such that $\FV{\vp} = \FV{\vp_1} \cup \FV{\vp_2}$ and
$\ha + \NN{\DNE{\Sigma_{k-1}}} \, (\ha$ if $k=0)$ proves
%\vdash 
$\neg \neg (\vp_1 \lor \vp_2) \lr \neg \neg \vp $.
\end{lemma}
\begin{proof}
Without loss of generality, assume $k>0$, $\vp_1 :\equiv \forall x \rho_1(x)$ and $\vp_2 :\equiv \forall y \rho_2(y)$ where $\rho_1(x), \rho_2(y) \in \Sigma_{k-1}$ (see Remark \ref{rem: sigma_k and pi_k}).
By Lemma \ref{lem: exists AvB}, it suffices to show
\begin{equation*}
\ha +     \NN{\DNE{\Sigma_{k-1}}} \vdash \neg \neg \left( \forall x \rho_1(x) \lor \forall y \rho_2(y)\right) \lr \neg \neg \forall x,y \left( \rho_1(x) \lor \rho_2 (y) \right).
\end{equation*}
The implication from the left to the right is straightforward.
The converse implication is shown as follows: 
%is is the case since
$$
\begin{array}{cl}
& \neg \neg \forall x,y \left( \rho_1(x) \lor \rho_2 (y) \right)
 \\
\underset{ \NN{\DNE{\Sigma_{k-1}}}}{\llr}&\neg \neg \forall x,y \left( \neg \neg \rho_1(x) \lor \neg \neg \rho_2 (y) \right)
 \\
\lra & \forall x,y \neg \neg  \left( \neg \neg \rho_1(x) \lor \neg \neg \rho_2 (y) \right)
 \\
\llr &
\neg \exists x,y \neg \left( \neg \neg \rho_1(x) \lor \neg \neg \rho_2 (y) \right)
 \\
\llr &  \neg \exists x,y \left( \neg \rho_1(x) \land \neg \rho_2 (y) \right)
 \\
 \llr &  \neg  \left(\neg \neg \exists x \neg \rho_1(x) \land \neg \neg \exists y \neg \rho_2 (y) \right)
 \\
  \llr &  \neg  \left(\neg \forall  x \neg \neg \rho_1(x) \land \neg \forall y \neg  \neg \rho_2 (y) \right)
 \\
 \llr&  \neg \neg  \left(\forall  x \neg \neg \rho_1(x) \lor \forall y \neg  \neg \rho_2 (y) \right)\\
 \underset{ \NN{\DNE{\Sigma_{k-1}}}}{\llr}&
  \neg \neg  \left(\forall  x  \rho_1(x) \lor \forall y \rho_2 (y) \right).
\end{array}
$$
\end{proof}

\begin{lemma}
\label{lem: NP and NS}
Let $k$ be a natural number.
\begin{enumerate}
    \item 
    \label{item: NP}
    For all $\vp \in \Pi_k$, there exists $\psi \in \Sigma_k$ such that
$\FV{\vp} = \FV{\psi}$ and
$\ha + \DNE{\Sigma_k} $ proves $ \neg \vp \lr \psi$.
\item
\label{item: NS}
For all $\vp \in \Sigma_k$, there exists $\psi \in \Pi_k$ such that
$\FV{\vp} = \FV{\psi}$ and
$\ha + \DNE{\Sigma_{k-1}} \, (\ha$ if $k=0) $ proves $\neg \vp \lr \psi$.
\end{enumerate}
\begin{proof}
By simultaneous induction on $k$.
The base case is trivial.
In what follows, we show the induction step for $k+1$.
%The induction step is shown by the simultaneous induction on the structure of formulas.

Let $\vp :\equiv \forall x \rho(x)$ where $\rho(x) \in \Sigma_{k}$.
By induction hypothesis, there exists $\rho'(x) \in \Pi_{k}$ such that
$\FV{\rho(x)} =\FV{\rho'(x)}$ and
$$
\ha + \DNE{\Sigma_{k-1}} \vdash \neg \rho(x)  \lr \rho'(x) .
$$
Then $\ha + \DNE{\Sigma_{k+1}}$ proves
$$
\begin{array}{cl}
     & \neg \forall x \rho(x) \\
\underset{\DNE{\Sigma_{k}}}{\llr}     &  \neg \forall x \neg \neg \rho(x)\\
\llr &  \neg \neg \exists x \neg \rho(x)\\
\underset{\text{[I.H.] } \DNE{\Sigma_{k-1}}}{\llr}&  \neg \neg \exists x \rho'(x)\\
\underset{\DNE{\Sigma_{k+1}}}{\llr}     &  \exists x \rho'(x),\\
\end{array}
$$
which is in $\Sigma_{k+1}$.

Next, let $\vp :\equiv \exists x \rho(x)$ where $\rho(x) \in \Pi_{k}$.
By induction hypothesis, there exists $\rho'(x) \in \Sigma_{k}$ such that
$\FV{\rho(x)} =\FV{\rho'(x)}$ and
$$
\ha + \DNE{\Sigma_{k}} \vdash \neg \rho(x)  \lr \rho'(x) .
$$
Then $\ha + \DNE{\Sigma_{k}}$ proves
$$
\neg \exists x \rho(x) \lr \forall x \neg  \rho(x)
\underset{\text{[I.H.] } \DNE{\Sigma_{k}}}{\llr}
\forall x  \rho'(x),
$$
which is in $\Pi_{k+1}$.
\end{proof}

\end{lemma}

\begin{lemma}
\label{lem: NPk <-> NNSk}
Let $k$ be a natural number.
For all $\vp \in \Pi_k$, there exists $\psi \in \Sigma_k$ such that
$\FV{\vp} = \FV{\psi}$ and
$\ha + \NN{\DNE{\Sigma_k}} \vdash  \neg \vp \lr \neg \neg \psi$.
\end{lemma}
\begin{proof}
Let $\vp \in \Pi_k$.
%It suffices to show that
By Lemma \ref{lem: NP and NS}.\eqref{item: NP},
there exists $\psi \in \Sigma_k$ such that
$\FV{\vp} = \FV{\psi}$ and
$\ha + \DNE{\Sigma_k} \vdash \neg \vp \lr  \psi$.
By Corollary \ref{cor: HA+P |- vp_1 <-> vp_2 => HA+NNP |- NNvp_1 <-> NNvp_2}, we have
$\ha + \NN{\DNE{\Sigma_k}} \vdash \neg \vp \lr \neg \neg \psi$.
\end{proof}

\begin{lemma}
 \label{lem: Ukp-DNS -> NNSk-1-LEM}
%Let $k$ be a natural number greater than $0$.
$\ha + \DNS{\U_k^+} \vdash \NN{\LEM{\Sigma_{k-1}}}$ for each natural number $k>0$.
\end{lemma}
\begin{proof}
Fix an instance of $\NN{\LEM{\Sigma_{k-1}}}$
$$
\vp :\equiv \neg \neg \forall x \left( \vp_1(x) \lor \neg \vp_1(x) \right),
$$
where $\vp_1(x) \in \Sigma_{k-1}$.
Note $\left( \vp_1(x) \lor \neg \vp_1(x) \right) \in \F_{k-1} \subseteq \U_k^+$. 
%By Lemma \ref{lem: basic facts on our classes}, we have $\forall x \left( \vp_1(x) \lor \neg \vp_1(x) \right) \in \U_k$.
Since $\ha $ proves $\forall x \neg \neg \left( \vp_1(x) \lor \neg \vp_1(x) \right)$,
we have that
$\ha + \DNS{\U_k^+}$ proves
$ \neg \neg \forall x \left( \vp_1(x) \lor \neg \vp_1(x) \right)$, namely, $\vp$. 
\end{proof}

\begin{corollary}
\label{cor: Ukp-DNS -> NNSk-1-DNE}
$\ha + \DNS{\U_k^+} \vdash \NN{\DNE{\Sigma_{k-1}}}$  for each natural number $k>0$.
\end{corollary}
\begin{proof}
Immediate from Remark \ref{rem: NNDNS <-> DNS}, Lemma \ref{lem: Ukp-DNS -> NNSk-1-LEM} and the fact that $\LEM{\Sigma_{k-1}}$ implies $\DNE{\Sigma_{k-1}}$.
\end{proof}

\begin{remark}
\label{rem: UkpDNS <-> UkDNS}
By Remark \ref{rem: Ekp and Ukp}, $\DNS{\U_k^+}$ is equivalent to $\DNS{\U_k}$ over $\ha$.
Then $\DNS{\U_k^+}$ can be replaced by $\DNS{\U_k}$ throughout the paper.
\end{remark}

\section{Prenex normal form theorems}\label{sec: PNFT}
In this section, we show the modified version of \cite[Theorem 2.7]{ABHK04}.
Prior to that, we first show a variant of the prenex normal form theorem:
\begin{lemma}
\label{PNFT_of_NE_k_and_NNU_k}
For each natural number $k$ and a formula $\varphi$ (possibly containing free variables),
if $\varphi \in \U_k^+$, then there exists $\varphi' \in \Pi_k$ such that $\FV{\vp}=\FV{\vp'}$ and
    $$
    \ha + \DNS{\U_k^+}  \vdash \neg \neg \varphi \leftrightarrow \neg \neg \varphi'.$$
    
%the following hold:
\end{lemma}
\begin{proof}
By simultaneous induction on $k$, we show the following two statements (which are in fact equivalent):
\begin{enumerate}
    \item 
    \label{eq: item for NE_k in PNFT}
    if $\varphi \in \E_k^+$, then there exists $\varphi' \in \Pi_k$ such that $\FV{\vp}=\FV{\vp'}$ and
    $$
    \ha + \DNS{\U_k^+}  \vdash \neg \varphi \leftrightarrow \neg \neg \varphi'; 
    $$
        \item 
        \label{eq: item for NNU_k in PNFT}
    if $\varphi \in \U_k^+$, then there exists $\varphi' \in \Pi_k$ such that $\FV{\vp}=\FV{\vp'}$ and
    $$
    \ha + \DNS{\U_k^+}  \vdash \neg \neg \varphi \leftrightarrow \neg \neg \varphi'.$$
\end{enumerate}
The base case is trivial (one can take $\varphi'$ as $\varphi$ itself).
In what follows, we show the induction step.

For the induction step, assume the items \ref{eq: item for NE_k in PNFT} and \ref{eq: item for NNU_k in PNFT} for $k-1$.
%The induction step is shown by induction on the structure of formulas.
%Fix $k>0$ and assume that for $k'<k$ and a formula $\psi$,
We show the items \ref{eq: item for NE_k in PNFT} and \ref{eq: item for NNU_k in PNFT} for $k$
%and \eqref{eq: IH for not Sigma_k in PNFT} in Theorem \ref{thm: PNFT} 
simultaneously by induction on the structure of formulas.
When $\varphi$ is a  prime  formula, by Lemma \ref{lem: sigma_k and pi_k}, we have $\varphi'$ which satisfies the requirement.
For the induction step, assume that the items \ref{eq: item for NE_k in PNFT} and \ref{eq: item for NNU_k in PNFT} hold for $\vp_1$ and $\vp_2$.
When it is clear from the context, we suppress the argument on free variables.
%On the other hand, 
%Note that
%$\NN{\DNE{\Sigma_{k-1}}}$ is provable in $\ha + \DNS{\U_k^+}$ since $\left( \neg \neg \vp (x) \to \vp(x)\right) \in \U_k^+$ for $\vp(x) \in \Sigma_{k-1}$.

The case of $\vp_1 \land \vp_2$:
First, assume $\vp_1 \land \vp_2 \in \E_k^+$.
By Lemma \ref{lem: basic facts on our classes}, we have $\vp_1,\vp_2 \in \E_k^+$.
By induction hypothesis, there exist $\vp_1' ,\vp_2' \in \Pi_k$ such that
$\ha + \DNS{\U_k^+} \vdash (\neg \vp_1 \leftrightarrow \neg \neg \vp_1' ) \land ( \neg \vp_2 \leftrightarrow \neg \neg \vp_2')$.
%\eqref{eq: IH for Pi_k in PNFT}
%Since there is a $\Pi_k$ 
%By Lemma \ref{lem: basic facts on arithmetical hierarchy}
By Lemma \ref{lem: HA+NNS_{k-1}-DNE |- NN(Pk v Pk) <-> NNPk}, there exists $\vp' \in \Pi_k$ such that $\FV{\vp'} = \FV{\vp_1'} \cup \FV{\vp_2'}$ and
$$\ha +\NN{\DNE{\Sigma_{k-1}}} \vdash \neg \neg (\vp_1' \lor \vp_2') \lr \neg \neg \vp'.$$
%\vp' \lr \vp_1'\land \vp_2' \lr \vp_1 \land \vp_2.$$
%In the same manner, if 
By Corollary \ref{cor: Ukp-DNS -> NNSk-1-DNE}, $\ha + \DNS{\U_k^+}$ proves
$$
\begin{array}{cl}
& \neg (\vp_1 \land \vp_2) \\
\llr & \neg  (\neg \neg \vp_1 \land \neg \neg \vp_2) \\
\underset{\text{[I.H.] } \DNS{\U_k^+}}{\llr} &  \neg  (\neg \vp_1' \land  \neg \vp_2') \\
\llr & \neg \neg (\vp_1' \lor  \neg \vp_2') \\
\underset{\NN{\DNE{\Sigma_{k-1}}}}{\llr} & \neg \neg \vp'.
\end{array}
$$

Next, assume $\vp_1 \land \vp_2 \in \U_k^+$.
By Lemma \ref{lem: basic facts on our classes}, we have $\vp_1,\vp_2 \in \U_k^+$.
By induction hypothesis, there exist $\vp_1' ,\vp_2' \in \Pi_k$ such that
$\ha + \DNS{\U_k^+} \vdash (\neg \neg \vp_1 \leftrightarrow \neg \neg \vp_1' ) \land ( \neg \neg \vp_2 \leftrightarrow \neg \neg \vp_2')$.
By Lemma \ref{lem: basic facts on arithmetical hierarchy}, there exists $\vp' \in \Pi_k$ such that $\FV{\vp'} = \FV{\vp_1'} \cup \FV{\vp_2'}$ and $\ha \vdash \vp' \lr \vp_1' \land \vp_2'$.
Then $\ha + \DNS{\U_k^+}$ proves
$$
\neg \neg (\vp_1 \land \vp_2) \lr \neg \neg \vp_1 \land \neg \neg \vp_2 \underset{\text{[I.H.] }\DNS{\U_k^+} }{\llr}  \neg \neg \vp_1' \land \neg \neg \vp_2' \lr \neg \neg (\vp_1' \land \vp_2') \lr \neg \neg \vp'.
$$

The case of $\vp_1 \lor \vp_2$:
First, assume $\vp_1 \lor \vp_2 \in \E_k^+$.
By Lemma \ref{lem: basic facts on our classes}, we have $\vp_1,\vp_2 \in \E_k^+$.
By induction hypothesis, there exist $\vp_1' ,\vp_2' \in \Pi_k$ such that
$\ha + \DNS{\U_k^+} \vdash (\neg \vp_1 \leftrightarrow \neg \neg \vp_1' ) \land ( \neg \vp_2 \leftrightarrow \neg \neg \vp_2')$.
By Lemma \ref{lem: basic facts on arithmetical hierarchy}, there exists $\vp' \in \Pi_k$ such that $\FV{\vp'} = \FV{\vp_1'} \cup \FV{\vp_2'}$ and $\ha \vdash \vp' \lr \vp_1' \land \vp_2'$.
Then $\ha + \DNS{\U_k^+}$ proves
$$
\begin{array}{cl}
& \neg (\vp_1 \lor \vp_2)\\
\llr & \neg \vp_1 \land \neg \vp_2\\
\underset{\text{[I.H.] }\DNS{\U_k^+} }{\llr} & \neg \neg \vp_1' \land  \neg \neg \vp_2'\\
\llr & \neg \neg (\vp_1' \land \vp_2')\\
\llr & \neg \neg \vp'.
\end{array}
$$

Next,  assume $\vp_1 \lor \vp_2 \in \U_k^+$.
By Lemma \ref{lem: basic facts on our classes}, we have $\vp_1,\vp_2 \in \U_k^+$.
By induction hypothesis, there exist $\vp_1' ,\vp_2' \in \Pi_k$ such that
$\ha + \DNS{\U_k^+} \vdash (\neg \neg \vp_1 \leftrightarrow \neg \neg \vp_1' ) \land ( \neg \neg \vp_2 \leftrightarrow \neg \neg \vp_2')$.
By Lemma \ref{lem: HA+NNS_{k-1}-DNE |- NN(Pk v Pk) <-> NNPk}, there exists $\vp' \in \Pi_k$ such that $\FV{\vp'} = \FV{\vp_1'} \cup \FV{\vp_2'}$ and
$$\ha +\NN{\DNE{\Sigma_{k-1}}} \vdash \neg \neg (\vp_1' \lor \vp_2') \lr \neg \neg \vp'.$$
By Corollary \ref{cor: Ukp-DNS -> NNSk-1-DNE}, $\ha + \DNS{\U_k^+}$ proves
$$
\begin{array}{cl}
& \neg \neg (\vp_1 \lor \vp_2)\\
\llr& \neg \neg ( \neg \neg \vp_1 \lor \neg \neg \vp_2)\\
\underset{\text{[I.H.] }\DNS{\U_k^+} }{\llr} & \neg \neg ( \neg \neg \vp_1' \lor \neg \neg \vp_2')\\
\llr &\neg \neg ( \vp_1' \lor \vp_2')\\
\underset{\NN{\DNE{\Sigma_{k-1}}}}{\llr}& \neg \neg \vp'. 
\end{array}
$$

The case of $\vp_1 \to \vp_2$:
First, assume $\vp_1 \to \vp_2 \in \E_k^+$.
By Lemma \ref{lem: basic facts on our classes}, we have $\vp_1 \in \U_k^+ $ and $\vp_2 \in \E_k^+$.
By induction hypothesis, there exist $\vp_1' ,\vp_2' \in \Pi_k$ such that
$\ha + \DNS{\U_k^+} $ proves  $\neg \neg \vp_1 \leftrightarrow \neg \neg \vp_1' $ and $ \neg \vp_2 \leftrightarrow \neg \neg \vp_2'$.
By Lemma \ref{lem: basic facts on arithmetical hierarchy}, there exists $\vp' \in \Pi_k$ such that $\FV{\vp'} = \FV{\vp_1'} \cup \FV{\vp_2'}$ and $\ha \vdash \vp' \lr \vp_1' \land \vp_2'$.
Then $\ha + \DNS{\U_k^+}$ proves
$$
\begin{array}{cl}
& \neg (\vp_1 \to \vp_2)\\
\llr & \neg \neg \vp_1 \land \neg \vp_2\\
\underset{\text{[I.H.] }\DNS{\U_k^+} }{\llr} & \neg \neg \vp_1' \land \neg \neg \vp_2'\\
\llr & \neg \neg (\vp_1' \land \vp_2')\\
\llr & \neg \neg \vp'.
\end{array}
$$

Next, assume $\vp_1 \to \vp_2 \in \U_k^+$.
By Lemma \ref{lem: basic facts on our classes}, we have $\vp_1 \in \E_k^+ $ and $\vp_2 \in \U_k^+$.
By induction hypothesis, there exist $\vp_1' ,\vp_2' \in \Pi_k$ such that
$\ha + \DNS{\U_k^+}$ proves $\neg \vp_1 \leftrightarrow \neg \neg \vp_1'  $ and $ \neg \neg \vp_2 \leftrightarrow \neg \neg \vp_2'$.
By Lemma \ref{lem: HA+NNS_{k-1}-DNE |- NN(Pk v Pk) <-> NNPk}, there exists $\vp' \in \Pi_k$ such that $\FV{\vp'} = \FV{\vp_1'} \cup \FV{\vp_2'}$ and
$$\ha +\NN{\DNE{\Sigma_{k-1}}} \vdash \neg \neg (\vp_1' \lor \vp_2') \lr \neg \neg \vp'.$$
By Corollary \ref{cor: Ukp-DNS -> NNSk-1-DNE}, $\ha + \DNS{\U_k^+}$ proves
$$
\begin{array}{cl}
& \neg \neg (\vp_1 \to \vp_2)\\
\llr & \neg (\neg \neg \vp_1 \land \neg \vp_2)\\
\underset{\text{[I.H.] }\DNS{\U_k^+} }{\llr} & \neg( \neg \vp_1' \land \neg \vp_2')\\
\llr & \neg \neg (\vp_1' \lor \vp_2')\\
\underset{\NN{\DNE{\Sigma_{k-1}}} }{\llr} &  \neg \neg \vp'.
\end{array}
$$

The case of $\forall x \vp_1(x)$:
First, assume $\forall x \vp_1(x) \in \E_k^+$.
By Lemma \ref{lem: basic facts on our classes}, we have $\forall x \vp_1(x) \in \U_{k-1}^+$.
By the item \ref{eq: item for NNU_k in PNFT} for $k-1$, there exists $\vp' \in \Pi_{k-1}$ such that
$$
\ha + \DNS{\U_{k-1}^+} \vdash \neg \neg \forall x \vp_1(x)  \lr \neg \neg \vp'.
$$
By Lemma \ref{lem: NPk <-> NNSk}, there exists $\vp''\in \Sigma_{k-1} \subseteq \Pi_k$ (see Remark \ref{rem: identification on Sigma_k and Pi_k}) such that $\FV{\vp'}= \FV{\vp''}$ and
$$
\ha +\NN{\DNE{\Sigma_{k-1}}} \vdash \neg  \vp' \lr \neg \neg \vp''.
$$
By Corollary \ref{cor: Ukp-DNS -> NNSk-1-DNE}, $\ha + \DNS{\U_k^+}$ proves
$$
\neg \forall x \vp_1(x)  \underset{\text{[I.H.] }\DNS{\U_{k-1}^+} }{\llr}  \neg  \vp' \underset{\NN{\DNE{\Sigma_{k-1}}} }{\llr} \neg \neg \vp'' .
$$

Next, assume $\forall x \vp_1(x) \in \U_k^+$.
By Lemma \ref{lem: basic facts on our classes}, we have $\vp_1(x) \in \U_{k}^+$.
By induction hypothesis, there exists $\vp_1'(x) \in \Pi_{k}$ such that
$$
\ha + \DNS{\U_{k}^+} \vdash \neg \neg  \vp_1(x)  \lr \neg \neg \vp_1' (x).
$$
Then $\ha + \DNS{\U_k^+}$ proves
$$
 \neg \neg \forall x \vp_1(x)  
 \underset{\DNS{\U_{k}^+}}{\llr}  \forall x \neg \neg  \vp_1(x)   
\underset{\text{[I.H.] }\DNS{\U_{k}^+} }{\llr} \forall x \neg \neg  \vp_1'(x)
 \underset{\DNS{\U_{k}^+}}{\llr}   \neg \neg \forall x \vp_1'(x).   
$$

The case of $\exists x \vp_1(x)$:
First, assume $\exists x \vp_1(x) \in \E_k^+$.
By Lemma \ref{lem: basic facts on our classes}, we have $\vp_1(x) \in \E_{k}^+$.
By induction hypothesis, there exists $\vp_1'(x) \in \Pi_{k}$ such that
$$
\ha + \DNS{\U_{k}^+} \vdash \neg \vp_1(x) \lr \neg \neg \vp_1'(x).
$$
Then $\ha + \DNS{\U_k^+}$ proves
$$
 \neg \exists x \vp_1(x) 
 \lr  \forall x  \neg \vp_1(x) 
\underset{\text{[I.H.] }\DNS{\U_{k}^+} }{\llr} \forall x \neg \neg  \vp_1'(x)
 \underset{\DNS{\U_{k}^+}}{\llr}   \neg \neg \forall x \vp_1'(x).
 $$
 
 Next, assume  that $\exists x \vp_1(x) \in \U_k^+$.
By Lemma \ref{lem: basic facts on our classes}, we have $\exists x \vp_1(x) \in \E_{k-1}^+$.
By the item \ref{eq: item for NE_k in PNFT} for $k-1$, there exists $\vp' \in \Pi_{k-1}$ such that
$$
\ha + \DNS{\U_{k-1}^+} \vdash \neg  \exists x \vp_1(x)  \lr \neg \neg \vp'.
$$
By Lemma \ref{lem: NPk <-> NNSk}, there exists $\vp''\in \Sigma_{k-1} \subseteq \Pi_k$ (see Remark \ref{rem: identification on Sigma_k and Pi_k}) such that $\FV{\vp'}= \FV{\vp''}$ and
$$
\ha +\NN{\DNE{\Sigma_{k-1}}} \vdash \neg  \vp' \lr \neg \neg \vp''.
$$
By Corollary \ref{cor: Ukp-DNS -> NNSk-1-DNE}, $\ha + \DNS{\U_k^+}$ proves
$$
 \neg \neg  \exists x \vp_1(x)  \underset{\text{[I.H.] }\DNS{\U_{k-1}^+} }{\llr} \neg  \vp'\underset{\NN{\DNE{\Sigma_{k-1}}}  }{\llr} \neg \neg \vp''.
$$
\end{proof}

%\begin{remark}
%\label{rem: eq on NE_k and NNU_k}
%\eqref{eq: item for NE_k in PNFT} and \eqref{eq: item for NNU_k in PNFT} in Lemma \ref{PNFT_of_NE_k_and_NNU_k} are equivalent.
%\end{remark}

The following lemma is used a lot of times implicitly in the proof of our prenex normal form theorem (Theorem \ref{thm: PNFT}).
\begin{lemma}[cf. Fact 2.2 in \cite{ABHK04}]
\label{lem: k -> k-1}
Let $k$ be a natural number.
\begin{enumerate}
    \item 
    \label{item: s_k+1-DNE -> (p_k v p_k)-DNE}
    $\ha + \DNE{\Sigma_{k+1}} \vdash \DNE{(\Pi_{k} \lor \Pi_{k})}$.
    \item
        \label{item: (p_k+1 v p_k+1)-DNE -> s_k-DNE}
$\ha +\DNE{(\Pi_{k+1} \lor \Pi_{k+1})} \vdash \DNE{\Sigma_{k}}$.   
  \item
        \label{item: NN(p_k+1 v p_k+1)-DNE -> NNs_k-DNE}
$\ha +\NN{\DNE{(\Pi_{k+1} \lor \Pi_{k+1})}} \vdash \NN{\DNE{\Sigma_{k}}}$.    

\item
\label{item: sk-DNE -> pk+1-DNE}
$
\ha +  \DNE{\Sigma_{k}} \vdash \DNE{\Pi_{k+1}}
$.
%\label{item: }
\end{enumerate}
\end{lemma}
\begin{proof}
\eqref{item: s_k+1-DNE -> (p_k v p_k)-DNE}:
For formulas $\vp_1$ and $\vp_2$ in $\Pi_{k}$, $\vp_1 \lor \vp_2$ is equivalent (over $\ha$) to
\begin{equation*}\label{eq: p1 v p1 without v}
\exists k \left(\left(k=0 \to \vp_1 \right) \land \left( k\neq 0 \to \vp_2 \right) \right),
\end{equation*}
which is equivalent to some $\vp \in \Sigma_{k+1}$ such that $\FV{\vp}= \FV{\vp_1 \lor \vp_2}$ over $\ha$ by Lemma \ref{lem: basic facts on arithmetical hierarchy}.\eqref{item: Pi land}.
Therefore any instance of $\DNE{(\Pi_{k} \lor \Pi_{k})}$ is derived from some instance of $\DNE{\Sigma_{k+1}}$.

        \eqref{item: (p_k+1 v p_k+1)-DNE -> s_k-DNE}:
        %Straightforward
        Any instance of $\DNE{\Sigma_{k}}$ is derived from some instance of 
        $\DNE{(\Pi_{k+1} \lor \Pi_{k+1})}$ since $\vp \in \Sigma_k$ is equivalent to $\forall y \vp \in \Pi_{k+1}$ with a variable $y$ not occurring freely in $\vp$ (cf. Lemma \ref{lem: sigma_k and pi_k}). 

\eqref{item: NN(p_k+1 v p_k+1)-DNE -> NNs_k-DNE}:
Immediate from \eqref{item: (p_k+1 v p_k+1)-DNE -> s_k-DNE} and Corollary \ref{cor: HA+P |- vp_1 <-> vp_2 => HA+NNP |- NNvp_1 <-> NNvp_2}.

        %dummy quantifier $\forall x$ in front.
\eqref{item: sk-DNE -> pk+1-DNE}:
Note that $\neg \neg \forall x \vp (x)$ implies  $\neg \neg \forall x \neg \neg \vp (x)$, which is intuitionistically equivalent to $\forall x \neg \neg \vp (x)$.
Then any instance of $\DNE{\Pi_{k+1}}$ is derived from some instance of $\DNE{\Sigma_{k}}$.
\end{proof}

We are now ready to show the modified version of \cite[Theorem 2.7]{ABHK04}.

%The following is our prenex normal form theorem, which is a modified variant of \cite[Theorem 2.7]{ABHK04} and is shown to be optimal in Section \ref{sec: Optimality}.
\begin{theorem}
\label{thm: PNFT}
For each natural number $k$ and a formula $\varphi$ (possibly containing free variables),
the following hold:
\begin{enumerate}
        \item 
        \label{eq: item for Sigma_k in PNFT}
    if $\varphi \in \E_k^+$, then
        there exists $\varphi' \in \Sigma_k$  such that $\FV{\vp}=\FV{\vp'}$ and
    $$
    \ha  + \DNE{\Sigma_k} +\DNS{\U_k^+} \vdash \varphi \leftrightarrow \varphi';
    $$
     \item 
    \label{eq: item for Pi_k in PNFT}
    if $\varphi \in \U_k^+$, then there exists $\varphi' \in \Pi_k$ such that $\FV{\vp}=\FV{\vp'}$ and
    $$
    \ha + \DNE{(\Pi_k\lor \Pi_k)}  \vdash \varphi \leftrightarrow \varphi' .
    $$
%    where $\DNE{\Sigma_{k-1}} $ is omitted if $k=0$.
\end{enumerate}
\end{theorem}
\begin{proof}
For the proof, we prepare the following auxiliary assertion (which is in fact a consequence from the item \ref{eq: item for Pi_k in PNFT}):
\begin{enumerate}
\setcounter{enumi}{2}
\item 
    \label{eq: item for NNPi_k in PNFT}
    if $\varphi \in \E_k^+$, then there exists $\varphi' \in \Pi_k$ such that $\FV{\vp}=\FV{\vp'}$ and
    $$
    \ha + \NN{\DNE{(\Pi_k\lor \Pi_k)}}  \vdash \neg \varphi \leftrightarrow \neg \neg \varphi'.
    $$
    \end{enumerate}
We show the items \ref{eq: item for Sigma_k in PNFT}, \ref{eq: item for Pi_k in PNFT} and 
\ref{eq: item for NNPi_k in PNFT} by induction on $k$ simultaneously.
The base case is trivial (one can take $\vp'$ as $\vp$ itself).  In what follows, we show the induction step.

Assume the items \ref{eq: item for Sigma_k in PNFT}, \ref{eq: item for Pi_k in PNFT} and \ref{eq: item for NNPi_k in PNFT} for $k-1$.
Since $\ha + \DNE{\Pi_{k-1}} \vdash \DNS{\Pi_{k-1}}$, by the item \ref{eq: item for Pi_k in PNFT} for $k-1$, we have
\begin{equation}
\label{eq: HA + Pk-1Pk-1DNE |- Uk-1pDNS}
\ha + \DNE{\left(\Pi_{k-1}\lor \Pi_{k-1}\right)} \vdash \DNS{\U_{k-1}^+}.
\end{equation}
We show the items \ref{eq: item for Sigma_k in PNFT}, \ref{eq: item for Pi_k in PNFT} and \ref{eq: item for NNPi_k in PNFT} simultaneously by induction on the structure of formulas.
When $\varphi$ is a  prime  formula, by Lemma \ref{lem: sigma_k and pi_k}, we have $\varphi'$ which satisfies the requirement.
For the induction step, assume that the items \ref{eq: item for Sigma_k in PNFT}, \ref{eq: item for Pi_k in PNFT} and \ref{eq: item for NNPi_k in PNFT} hold for $\vp_1$ and $\vp_2$.
When it is clear from the context, we suppress the argument on free variables.

%$\NN{\DNE{\Sigma_{k-1}}}$ is provable in $\ha + \DNS{\U_k^+}$ since
%$\left( \neg \neg \vp (x) \to \vp(x)\right) \in \U_k^+$ for $\vp(x) \in \Sigma_{k-1}$.

The case of $\vp_1 \land \vp_2$:
For the second item, assume $\vp_1 \land \vp_2 \in \U_k^+$.
By Lemma \ref{lem: basic facts on our classes}, we have $\vp_1,\vp_2 \in \U_k^+$.
By induction hypothesis, there exist $\vp_1' ,\vp_2' \in \Pi_k$ such that
$\ha +\DNE{(\Pi_k\lor \Pi_k)} $ proves $\vp_1 \lr \vp_1'$ and $\vp_2 \lr \vp_2'$.
%\vdash (\vp_1 \leftrightarrow \vp_1' ) \land ( \vp_2 \leftrightarrow \vp_2')$.
By Lemma \ref{lem: basic facts on arithmetical hierarchy}, there exists $\vp' \in \Pi_k$ such that $\ha \vdash \vp' \lr \vp_1'\land \vp_2'$.
Then $\ha +\DNE{(\Pi_k\lor \Pi_k)}$ proves
$$\vp_1 \land \vp_2
\underset{\text{[I.H.] }\DNE{(\Pi_k\lor \Pi_k)} }{\llr} \vp_1' \land \vp_2' \lr \vp'.$$

For the first and third items, assume $\vp_1 \land \vp_2 \in \E_k^+$.
By Lemma \ref{lem: basic facts on our classes}, we have $\vp_1,\vp_2 \in \E_k^+$.
%By induction hypothesis, there exist $\vp_1' ,\vp_2' \in \Sigma_k$ such that
%$\ha +\DNE{\Sigma_k} + \DNS{\U_k^+} \vdash (\vp_1 \leftrightarrow \vp_1' ) \land ( \vp_2 \leftrightarrow \vp_2')$.
Then we have $\vp'\in \Sigma_k$ such that $ \ha +\DNE{\Sigma_k} + \DNS{\U_k^+} \vdash \vp_1 \land \vp_2 \lr \vp'$ as in the second item.
%case of $\vp_1 \land \vp_2 \in \U_k^+$.
On the other hand, by induction hypothesis, there exist  $\vp_1'' ,\vp_2'' \in \Pi_k$ such that
$\ha + \NN{\DNE{(\Pi_k\lor \Pi_k)}} $ proves  $\neg \vp_1 \leftrightarrow \neg \neg \vp_1''  $ and  $ \neg \vp_2 \leftrightarrow \neg \neg  \vp_2''$.
In addition, by Lemma \ref{lem: HA+NNS_{k-1}-DNE |- NN(Pk v Pk) <-> NNPk}, there exists $\vp'' \in \Pi_k$ such that
$\ha + \NN{\DNE{\Sigma_{k-1}}} \vdash \neg \neg \vp'' \leftrightarrow \neg \neg (\vp_1'' \lor \vp_2'')$.
Then, by Lemma \ref{lem: k -> k-1}, we have that $\ha +  \NN{\DNE{(\Pi_k\lor \Pi_k)}}$ proves
%by Corollary \ref{cor: Ukp-DNS -> NNSk-1-DNE}, 
$$
\begin{array}{cl}
&\neg (\vp_1 \land \vp_2) \\
\llr& \neg (\neg \neg \vp_1 \land \neg \neg  \vp_2)\\
\llr& \neg \neg (\neg \vp_1 \lor \neg \vp_2)\\
\underset{\text{[I.H.] }\NN{\DNE{(\Pi_{k}\lor \Pi_{k})}}}{\llr} & \neg \neg (\neg \neg \vp_1'' \lor \neg \neg \vp_2'' )\\
\llr& \neg (\neg \vp_1'' \land \neg \vp_2'')\\
\llr&\neg \neg (\vp_1'' \lor \vp_2'')\\
\underset{\NN{\DNE{\Sigma_{k-1}}}}{\llr}&\neg \neg \vp''.
\end{array}
$$

The case of  $\vp_1 \lor \vp_2$:
For the second item, assume $\vp_1 \lor \vp_2 \in \U_k^+$.
By Lemma \ref{lem: basic facts on our classes}, we have $\vp_1,\vp_2 \in \U_k^+$.
Then, by induction hypothesis, there exist $\rho_1(x_1), \rho_2(x_2) \in \Sigma_{k-1}$ such that
$\ha + \DNE{(\Pi_k\lor \Pi_k)}  $ proves  $\vp_1 \lr \forall x_1 \rho_1(x_1) $ and  $\vp_2 \lr \forall x_2 \rho_2(x_2) $.
By Lemma \ref{lem: k -> k-1}, $\ha + \DNE{(\Pi_k\lor \Pi_k)}   $ proves
$$
\begin{array}{cl}
%&\vp_1 \lor \vp_2\\
%\llr
& \forall x_1 \rho_1(x_1) \lor \forall x_2 \rho_2(x_2)\\
\lra& \forall x_1, x_2 \left( \rho_1(x_1) \lor \rho_2(x_2)\right)\\
\lra& \neg \left( \exists x_1 \neg \rho_1 (x_1) \land  \exists x_2 \neg \rho_2 (x_2) \right) \\
\llr& \neg \left( \neg \neg  \exists x_1 \neg \rho_1 (x_1) \land  \neg \neg  \exists x_2 \neg \rho_2 (x_2) \right) \\
\underset{\DNE{\Sigma_{k-1}}}{\llr} &\neg \left( \neg \forall x_1 \rho_1 (x_1) \land  \neg \forall x_2 \rho_2 (x_2) \right) \\
\llr & \neg \neg \left( \forall x_1 \rho_1 (x_1) \lor  \forall x_2 \rho_2 (x_2) \right) \\
\underset{\DNE{(\Pi_k\lor \Pi_k)}}{\lra} &\forall x_1 \rho_1 (x_1) \lor  \forall x_2 \rho_2 (x_2).
\end{array}
$$
By Lemma \ref{lem: exists AvB},
%andhypothesis of course of values induction, 
there exists  $\xi(x_1, x_2)  \in \Sigma_{k-1}$ such that $\ha \vdash \xi(x_1, x_2) \lr \rho_1(x_1) \lor \rho_2(x_2)$.
% $$\ha  + \DNE{\Sigma_{k-1}} + \NN{\DNE{(\Pi_{k-1}\lor \Pi_{k-1})}} \vdash \xi(x_1,x_2) \lr \rho_1(x_1) \lor \rho_2(x_2).$$
Then we have that $\ha + \DNE{(\Pi_k\lor \Pi_k)}  $ proves
$$
\begin{array}{cl}
&\vp_1 \lor \vp_2\\
\underset{\text{[I.H.] }\DNE{(\Pi_k\lor \Pi_k)}}{\llr}& \forall x_1 \rho_1 (x_1) \lor \forall x_2 \rho_2 (x_2)\\
\underset{\DNE{(\Pi_k\lor \Pi_k)}}{\llr} & \forall x_1, x_2 \left( \rho_1(x_1) \lor \rho_2(x_2)\right)\\
\llr& \forall x_1, x_2\,  \xi(x_1, x_2) \in \Pi_k.
\end{array}$$
%which is in $\Pi_k$.

%Then it suffices to show that $  \forall x_1 \rho_1(x_1) \lor \forall x_2 \rho_2(x_2)$ is equivalent to some
%$\Pi_k$-formula in $\ha + \DNE{(\Pi_k\lor \Pi_k)} + \NN{\DNE{\Sigma_k}}$.
For the first and third items, assume $\vp_1 \lor \vp_2 \in \E_k^+$.
By Lemma \ref{lem: basic facts on our classes}, we have $\vp_1,\vp_2 \in \E_k^+$.
By induction hypothesis, there exist $\rho_1(x_1), \rho_2(x_2) \in \Pi_{k-1}$ such that $\ha + \DNE{\Sigma_k} + \DNS{\U_k^+} $ proves  $\vp_1 \lr \exists x_1 \rho_1(x_1) $ and $ \vp_2 \lr \exists x_2 \rho_2(x_2) $.
%By Lemma \ref{lem: basic facts on arithmetical hierarchy} and (a),
%By hypothesis of course of values induction, 
By the item \ref{eq: item for Pi_k in PNFT} for $k-1$, there exists $\xi(x_1, x_2) \in \Pi_{k-1}$ such that  
 $$\ha  + \DNE{(\Pi_{k-1}\lor \Pi_{k-1})} \vdash \xi(x_1,x_2) \lr \rho_1(x_1) \lor \rho_2(x_2).$$
%Since $\vp_1 \lor \vp_2$ is equivalent to $\exists x_1,x_2 \left(\rho_1(x_1) \lor \rho_2(x_2)  \right)$, 
%we have that $\vp_1 \lor \vp_2$ is equivalent to $\exists x_1, x_2\,  \xi(x_1, x_2) \in \Sigma_k $ over $\ha + \DNE{\Sigma_k} + \DNS{\U_k^+} $ (note Lemma \ref{lem: k -> k-1}.\eqref{item: s_k+1-DNE -> (p_k v p_k)-DNE}).
By  Lemma
%\ref{lem: exists AvB} and 
\ref{lem: k -> k-1}, we have that $\ha + \DNE{\Sigma_k} + \DNS{\U_k^+} $ proves
$$
\begin{array}{cl}
&\vp_1 \lor \vp_2 \\
\underset{\text{[I.H.] }\DNE{\Sigma_k},\, \DNS{\U_k^+}}{\llr} & \exists x_1 \rho_1 (x_1) \lor \exists x_2 \rho_2 (x_2)\\
\llr &\exists x_1, x_2 (\rho_1(x_1) \lor \rho_2(x_2)) \\
\underset{\text{[I.H.] }\DNE{(\Pi_{k-1}\lor \Pi_{k-1})}}{\llr} &\exists x_1, x_2\, \xi(x_1, x_2).
\end{array}
$$
Thus we are done for the first item.
For the third item, by induction hypothesis, there exist  $\vp_1'' ,\vp_2'' \in \Pi_k$ such that
$\ha + \NN{\DNE{(\Pi_k\lor \Pi_k)}}$ proves $\neg \vp_1 \leftrightarrow \neg \neg \vp_1''$ and $ \neg \vp_2 \leftrightarrow \neg \neg  \vp_2''$.
In addition, by Lemma \ref{lem: basic facts on arithmetical hierarchy}, there exists $\vp'' \in \Pi_k$ such that
$\ha \vdash \vp'' \leftrightarrow \vp_1'' \land \vp_2''$.
Then $\ha +  \NN{\DNE{(\Pi_k\lor \Pi_k)}}$ proves
$$
\neg (\vp_1 \lor \vp_2 ) \leftrightarrow \neg \vp_1 \land \neg \vp_2 \underset{\text{[I.H.] }\NN{\DNE{(\Pi_{k}\lor \Pi_{k})}}}{\llr}  \neg \neg \vp_1'' \land \neg \neg \vp_2'' \leftrightarrow \neg \neg (\vp_1'' \land \vp_2'') \leftrightarrow \neg \neg \vp''.
$$

The case of $\vp_1 \to \vp_2$:
For the second item, assume $\vp_1 \to \vp_2 \in \U_k^+$.
By Lemma \ref{lem: basic facts on our classes}, we have $\vp_1 \in \E_k^+$ and $\vp_2 \in \U_k^+$.
By induction hypothesis, there exist  $ \rho_1(x_1), \rho_2(x_2)\in \Sigma_{k-1}$ such that
$\ha + \NN{\DNE{(\Pi_{k}\lor \Pi_{k})}} \vdash \neg \vp_1 \lr \neg \neg \forall x_1 \rho_1(x_1)$
and
$\ha + \DNE{(\Pi_{k}\lor \Pi_{k})}  \vdash \vp_2 \lr \forall x_2 \rho_2(x_2)$.
By Lemma \ref{lem: basic facts on our classes}, we have that $\neg \rho_1(x_1) \to \rho_2(x_2)$ is in $\E_{k-1}^+$.
Then,
%by hypothesis of course of values induction, 
by the item \ref{eq: item for Sigma_k in PNFT} for $k-1$,
there exists $\xi(x_1, x_2)  \in \Sigma_{k-1}$ such that
$$
\ha + \DNE{\Sigma_{k-1}} + \DNS{\U_{k-1}^+} \vdash 
\xi(x_1, x_2) \lr  \left( \neg \rho_1(x_1) \to \rho_2(x_2) \right).$$
%In addition, by the item \ref{eq: item for Pi_k in PNFT} for $k-1$, $\ha + \DNE{(\Pi_{k-1}\lor \Pi_{k-1})}$ proves $\neg \neg \psi \to \psi$ for any $\psi \in \U_{k-1}^+$, and hence, $\ha + \DNE{(\Pi_{k-1}\lor \Pi_{k-1})} \vdash \DNS{\U_{k-1}^+}$.
Then, using Lemma \ref{lem: k -> k-1} and \eqref{eq: HA + Pk-1Pk-1DNE |- Uk-1pDNS},
%and Corollary \ref{cor: pk v pk -DNE -> Ukp-DNE},
we have that $\ha + \DNE{(\Pi_{k}\lor \Pi_{k})} $ proves
$$
\begin{array}{cl}
&\vp_1 \to \vp_2\\
\underset{\text{[I.H.] }\DNE{(\Pi_{k}\lor \Pi_{k})}}{\llr} & \vp_1 \to \forall x_2 \rho_2 (x_2)\\
\underset{\DNE{\Pi_k}}{\llr} & \vp_1 \to \neg \neg \forall x_2 \rho_2 (x_2)\\
\llr &  \neg \neg \vp_1 \to  \neg \neg  \forall x_2 \rho_2 (x_2) \\
\underset{\text{[I.H.] }\NN{\DNE{(\Pi_{k}\lor \Pi_{k})}}}{\llr} & \neg \forall x_1 \rho_1(x_1) \to  \neg \neg  \forall x_2 \rho_2 (x_2)\\
\underset{\DNE{\Sigma_{k-1}} }{\llr} & \neg \neg \exists x_1 \neg \rho_1(x_1) \to  \neg \neg  \forall x_2 \rho_2 (x_2)\\
\llr & \neg \neg \forall x_1, x_2 \left( \neg \rho_1(x_1) \to  \rho_2 (x_2) \right)\\
\underset{\text{[I.H.] }\DNE{\Sigma_{k-1}} , \, \DNS{\U_{k-1}^+}  }{\llr} &\neg \neg \forall x_1, x_2 \, \xi( x_1, x_2)\\
\underset{\DNE{\Pi_k}}{\llr} &\forall x_1, x_2 \, \xi( x_1, x_2) \in \Pi_k.\\
\end{array}
$$

For the first and third items, assume $\vp_1 \to \vp_2 \in \E_k^+$.
By Lemma \ref{lem: basic facts on our classes}, we have $\vp_1 \in \U_k^+$ and $\vp_2 \in \E_k^+$.
By induction hypothesis, there exists
%$\rho_1(x_1)  \in \Sigma_{k-1}$ and 
$ \rho_2(x_2) \in \Pi_{k-1}$  such that
%\begin{itemize}
%\item
%$\ha +  \DNE{(\Pi_{k}\lor \Pi_{k})} \vdash \vp_1 \lr \forall x_1 \rho_1(x_1) $
%\item
$\ha + \DNE{\Sigma_{k}} +  \DNS{\U_k^+} \vdash  \vp_2 \lr  \exists x_2 \rho_2(x_2) $.
%$\ha + \DNE{(\Pi_k \lor \Pi_k)} \vdash  \vp_1 \lr \vp_1'$
%and
%$\ha + \NN{\DNE{(\Pi_k \lor \Pi_k)}} \vdash \neg \vp_2 \lr \neg \neg \vp_2'$.
%\end{itemize}
In addition, by Lemma \ref{PNFT_of_NE_k_and_NNU_k}, there exists $ \rho_1(x_1)\in \Sigma_{k-1}$ such that
$\ha + \DNS{\U_k^+} \vdash \neg \neg \vp_1 \lr \neg \neg \forall x_1 \rho_1(x_1)$.
%In particular, we have
%\begin{equation}\label{eq: NN PNFT_pk}
%    \ha +  \NN{\DNE{(\Pi_{k}\lor \Pi_{k})}}\vdash \neg \neg \vp_1 \lr \neg \neg \forall x_1 \rho_1(x_1)
%\end{equation} 
%from the former.
By Lemma \ref{lem: basic facts on our classes}, we have $\neg \rho_2(x_2) \to  \neg \rho_1(x_1)$ is in $\U_{k-1}^+$.
Then,
%by hypothesis of course of values induction, 
by the item \ref{eq: item for Pi_k in PNFT} for $k-1$,
there exists $\xi(x_1, x_2) \in \Pi_{k-1}$ such that
$$
\ha + \DNE{(\Pi_{k-1}\lor \Pi_{k-1})} \vdash \xi(x_1, x_2) \lr  \left(\neg \rho_2(x_2) \to  \neg \rho_1(x_1) \right).$$
Then, using Lemma \ref{lem: k -> k-1}, we have that $\ha + \DNE{\Sigma_{k}}+ \DNS{\U_k^+}$ proves
$$
\begin{array}{cl}
& \vp_1 \to \vp_2 \\
%\lra&  \neg \neg \vp_1 \to \neg \neg  \vp_2\\
\underset{\text{[I.H.] }\DNE{\Sigma_{k}}, \, \DNS{\U_k^+}}{\llr}& \vp_1 \to  \exists x_2 \rho_2(x_2)\\
\underset{\DNE{\Sigma_k}}{\llr}&  \vp_1 \to \neg \neg  \exists x_2 \rho_2(x_2)\\
\llr& \neg \neg \vp_1 \to \neg \neg  \exists x_2 \rho_2(x_2)\\
\underset{ \DNS{\U_k^+}}{\llr}&\neg \neg  \forall x_1 \rho_1(x_1) \to \neg \neg  \exists x_2 \rho_2(x_2)\\
\llr&\neg \neg \exists x_2 \left( \forall x_1 \rho_1(x_1) \to \rho_2(x_2) \right) \\
\underset{\DNE{\Pi_{k-1}}}{\llr}&\neg \neg \exists x_2 \left(\neg  \rho_2(x_2) \to \neg \forall x_1 \rho_1(x_1)  \right) \\
\underset{\DNE{\Sigma_{k-1}}}{\llr}&\neg \neg \exists x_2 \left(\neg  \rho_2(x_2) \to \neg \neg \exists x_1  \neg  \rho_1(x_1)  \right) \\
\llr &\neg \neg \exists x_2 \neg \neg \exists x_1  \left(\neg  \rho_2(x_2) \to  \neg  \rho_1(x_1)  \right) \\
\underset{\text{[I.H.] } \DNE{(\Pi_{k-1}\lor \Pi_{k-1})}}{\llr}&\neg \neg \exists x_1, x_2 \, \xi (x_1, x_2)\\
\underset{ \DNE{\Sigma_k}}{\llr}& \exists x_1, x_2 \, \xi (x_1, x_2) \in \Sigma_k .\\
\end{array}
$$
Thus we are done for the first item.
For the third item, by induction hypothesis, there exist
%$\rho_1(x_1)  \in \Sigma_{k-1}$ and 
$\vp_1', \vp_2' \in \Pi_k$ such that
%\begin{itemize}
%\item
%$\ha +  \DNE{(\Pi_{k}\lor \Pi_{k})} \vdash \vp_1 \lr \forall x_1 \rho_1(x_1) $
%\item
%$\ha + \DNE{\Sigma_{k}} +  \DNS{\U_k^+} \vdash  \vp_2 \lr  \exists x_2 \rho_2(x_2) $,
$\ha + \DNE{(\Pi_k \lor \Pi_k)} \vdash  \vp_1 \lr \vp_1'$
and
$\ha + \NN{\DNE{(\Pi_k \lor \Pi_k)}} \vdash \neg \vp_2 \lr \neg \neg \vp_2'$.
On the other hand, by Lemma \ref{lem: basic facts on arithmetical hierarchy}, there exists $\xi' \in \Pi_k$ such that $\ha \vdash \vp_1' \land \vp_2' \lr \xi'$.
Then $\ha +\NN{\DNE{(\Pi_k \lor \Pi_k)}} $ proves
$$
\neg( \vp_1 \to \vp_2) \lr (\neg \neg \vp_1 \land \neg \vp_2)\underset{\text{[I.H.] } \NN{\DNE{(\Pi_{k}\lor \Pi_{k})}}}{\llr} \neg \neg \vp_1' \land \neg \neg \vp_2' \lr \neg \neg (\vp_1' \land \vp_2') \lr \neg \neg \xi'.
$$

The case of $\forall x  \vp_1 (x )$:
For the second item, assume $\forall x  \vp_1 (x ) \in \U_k^+$.
By Lemma \ref{lem: basic facts on our classes}, we have $\vp_1 (x ) \in \U_k^+$.
By induction hypothesis, there exists $\vp_1'(x )  \in \Pi_{k}$ such that $\ha + \DNE{(\Pi_{k}\lor \Pi_{k})} \vdash \vp_1(x ) \lr \vp_1'(x )$.
Then $\forall x  \vp_1 (x )$ is equivalent to  $\forall x  \vp_1'(x ) \in \Pi_k$ over 
$\ha + \DNE{(\Pi_{k}\lor \Pi_{k})}  $.

For the first and third items, assume  $\forall x  \vp_1 (x ) \in \E_k^+$.
By Lemma \ref{lem: basic facts on our classes}, we have $\forall x \vp_1 (x ) \in \U_{k-1}^+$.
Then,
%by hypothesis of course of values induction,
by the item \ref{eq: item for Pi_k in PNFT} for $k-1$, there exists $\xi \in \Pi_{k-1} \subseteq \Sigma_k$ (see Remark \ref{rem: identification on Sigma_k and Pi_k}) such that
\begin{equation}
\label{eq: con of IH for k-1}
\ha + \DNE{(\Pi_{k-1}\lor \Pi_{k-1})} \vdash \forall x \vp_1(x ) \lr \xi.
\end{equation}
%Then $\ha + \NN{\DNE{(\Pi_{k-1}\lor \Pi_{k-1})}} \vdash \neg \neg \forall x \vp_1(x ) \lr \neg \neg \xi$.
%Note that $ \Pi_{k-1} \subseteq \Sigma_k \cup \Pi_k$ (see Remark \ref{rem: identification on Sigma_k and Pi_k}).
By Lemma \ref{lem: k -> k-1}, we are done for the first item.
For the third item, by Lemma \ref{lem: NP and NS}, there exists $\xi'\in \Sigma_{k-1} \subseteq \Pi_k$ (see Remark \ref{rem: identification on Sigma_k and Pi_k}) such that $\ha + \DNE{\Sigma_{k-1}} \vdash \neg \xi \lr \xi'$.
By Corollary \ref{cor: HA+P |- vp_1 <-> vp_2 => HA+NNP |- NNvp_1 <-> NNvp_2}, we have $\ha + \NN{\DNE{\Sigma_{k-1}}} \vdash \neg \xi \lr  \neg \neg \xi'$.
In addition, 
$$\ha + \NN{\DNE{(\Pi_{k-1}\lor \Pi_{k-1})}} \vdash \neg \neg  \forall x \vp_1(x ) \lr \neg \neg \xi $$
follows from \eqref{eq: con of IH for k-1}.
Then, by Lemma \ref{lem: k -> k-1}, we have that $\ha + \NN{\DNE{(\Pi_{k}\lor \Pi_{k})}}  $ proves
$$
\neg \forall x \vp_1(x ) \underset{\text{[I.H.] } \NN{\DNE{(\Pi_{k-1}\lor \Pi_{k-1})}}}{\llr} \neg \xi \underset{ \NN{\DNE{\Sigma_{k-1}}}}{\llr} \neg \neg \xi'.
$$
Thus we have shown the third item.

The case of $\exists x  \vp_1 (x )$:
%[Using $\DNS{\U_{k-1}^+}$]
For the second item, assume $\exists x  \vp_1 (x ) \in \U_k^+$.
By Lemma \ref{lem: basic facts on our classes}, we have $ \exists x \vp_1 (x ) \in \E_{k-1}^+$.
Then, 
by the item \ref{eq: item for Sigma_k in PNFT} for $k-1$, there exists $\xi  \in \Sigma_{k-1}\subseteq \Pi_k$ (see Remark \ref{rem: identification on Sigma_k and Pi_k}) such that $\ha +\DNE{\Sigma_{k-1}} + \DNS{\U_{k-1}^+} \vdash \exists x \vp_1(x ) \lr \xi$.
%In addition, by the item \ref{eq: item for Pi_k in PNFT} for $k-1$, $\ha + \DNE{(\Pi_{k-1}\lor \Pi_{k-1})}$ proves $\neg \neg \psi \to \psi$ for any $\psi \in \U_{k-1}^+$, and hence, $\ha + \DNE{(\Pi_{k-1}\lor \Pi_{k-1})} \vdash \DNS{\U_{k-1}^+}$.
By Lemma \ref{lem: k -> k-1} and \eqref{eq: HA + Pk-1Pk-1DNE |- Uk-1pDNS}, we have
$\ha + \DNE{(\Pi_{k}\lor \Pi_{k})} \vdash \exists x \vp_1(x ) \lr \xi$.

For the first and third items, assume $\exists x  \vp_1 (x ) \in \E_k^+$.
By Lemma \ref{lem: basic facts on our classes}, we have $\vp_1 (x ) \in \E_k^+$.
By induction hypothesis, there exist $\vp_1'(x )  \in \Sigma_{k}$ and $\vp_1''(x )  \in \Pi_k$ such that $\ha  +\DNE{\Sigma_{k}} +\DNS{\U_{k}^+} \vdash \vp_1(x ) \lr \vp_1'(x )$
and $\ha + \NN{\DNE{(\Pi_{k}\lor \Pi_{k})}} \vdash \neg \vp_1 (x) \lr \neg \neg \vp_1'' (x)$.
Then $\exists x  \vp_1 (x )$ is equivalent to $\exists x  \vp_1'(x ) \in \Sigma_k$ over 
$\ha +\DNE{\Sigma_{k}} +\DNS{\U_{k}^+} $.
Thus we are done for the first item.
For the third item, since $\ha + \NN{\DNE{(\Pi_{k}\lor \Pi_{k})}}$ proves
$$
\neg  \exists x \vp_1(x )\lr \forall x \neg  \vp_1(x )  \underset{ \text{[I.H.] }\NN{\DNE{(\Pi_{k}\lor \Pi_{k})}}}{\llr}   \forall x \neg  \neg \vp_1''(x ) ,
$$
we have that $\ha + \NN{\DNE{(\Pi_{k}\lor \Pi_{k})}}$ proves
%and hence,
$$
\neg  \exists x \vp_1(x )\lr \neg \neg   \forall x \neg  \neg \vp_1''(x ) .
$$
On the other hand, the latter is equivalent to $ \neg \neg   \forall x \vp_1''(x ) $ in the presence of $ \NN{\DNE{\Pi_{k}}}$.
Thus we have $\ha + \NN{\DNE{(\Pi_{k}\lor \Pi_{k})}} \vdash \neg  \exists x \vp_1(x )\lr \neg \neg   \forall x \vp_1''(x )$.
%we are done for the third item.
\end{proof}

\begin{comment}
\begin{corollary}
For each natural number $k$ and a formula $\varphi$ of $\ha$, %(possibly containing free variables),
 if $\varphi \in \U_k^+$, then there exists $\varphi' \in \Pi_k$ such that $\FV{\vp}=\FV{\vp'}$ and
    $$
    \ha +\DNE{\Sigma_{k-1}} + \NN{\DNE{(\Pi_k\lor \Pi_k)}}  \vdash \NN{\varphi} \leftrightarrow \varphi'.
    $$
\end{corollary}
\begin{proof}
Immediate from Theorem \ref{thm: PNFT}.\eqref{eq: item for NNPi_k in PNFT} and Lemma \ref{lem: k -> k-1}.\eqref{item: sk-DNE -> pk+1-DNE}.
\end{proof}
\end{comment}

\begin{corollary}
\label{cor: pk v pk -DNE -> Ukp-DNE}
$\ha + \NN{\DNE{(\Pi_{k}\lor \Pi_{k})}}  \vdash \DNS{\U_k^+} .$
\end{corollary}
\begin{proof}
Since $\DNS{\U_k^+} $ is intuitionistically equivalent to $\NN{\DNS{\U_k^+} }$ (see Remark \ref{rem: NNDNS <-> DNS}), it suffices to show $\ha + \DNE{(\Pi_{k}\lor \Pi_{k})}  \vdash \DNS{\U_k^+} $.
By Theorem \ref{thm: PNFT}.\eqref{eq: item for Pi_k in PNFT}, any formulas $\vp \in \U_k^+$ is equivalent to some  $\vp' \in \Pi_k$ such that $\FV{\vp} = \FV{\vp'}$ over $\ha + \DNE{(\Pi_{k}\lor \Pi_{k})}$.
%Then $ \DNS{\U_k^+} $ follows from $\DNE{\Pi_k}$.
%Thus
Since $\ha + \DNE{\Pi_k} \vdash \DNS{\Pi_k}$, we have $\ha + \DNE{(\Pi_{k}\lor \Pi_{k})}  \vdash \DNS{\U_k^+} $.
\end{proof}

\begin{remark}
Corollary \ref{cor: pk v pk -DNE -> Ukp-DNE} shows that
Lemma \ref{PNFT_of_NE_k_and_NNU_k} (equivalent to item \ref{eq: item for NE_k in PNFT} in the proof of Lemma \ref{PNFT_of_NE_k_and_NNU_k}) is a stronger statement of the item \ref{eq: item for NNPi_k in PNFT} in the proof of Theorem \ref{thm: PNFT}.
On the other hand, it is still open whether $\ha + \DNS{\U_k^+}$ is a proper subsystem of $\ha + \NN{\DNE{(\Pi_{k}\lor \Pi_{k})}} $.
\end{remark}

\begin{remark}
It follows from Theorem \ref{thm: PNFT} and the results in Section \ref{sec: CE} that $\ha + \DNE{\Sigma_1}$ does not prove $\DNS{\U_1^+}$.
\end{remark}

At the end of this section, we study the prenex normal form theorem for formulas which do not contain the disjunction $\lor$.
In fact, the proof of Theorem \ref{thm: PNFT} suggests that the unusual form $\DNE{(\Pi_k \lor \Pi_k)}$ of the double negation elimination is caused from the argument especially in the case of $\vp_1 \lor \vp_2$.
On the other hand, if a formula $\vp$ does not contain $\lor$, one can intuitionistically derive the original formula $\vp$ from a formula of the prenex normal form which is classically equivalent to $\vp$ (cf. \cite[Lemma 6.2.1]{vD13}).
Then the proof of the prenex normal form theorem for those formulas becomes to be fairly simple.

\begin{theorem}
\label{thm: PNFT for df-formulas}
For each natural number $k$ and a formula $\varphi$ (possibly containing free variables) which does not contain $\lor$,
the following hold:
\begin{enumerate}
        \item 
        \label{eq: item for Sigma_k in PNFT for df-formulas}
    if $\varphi \in \E_k^+$, then
        there exists $\varphi' \in \Sigma_k$  such that $\FV{\vp}=\FV{\vp'}$ and
    $$
    \ha  + \DNE{\Sigma_k} \vdash \varphi \leftrightarrow \varphi';
    $$
     \item 
    \label{eq: item for Pi_k in PNFT for df-formulas}
    if $\varphi \in \U_k^+$, then there exists $\varphi' \in \Pi_k$ such that $\FV{\vp}=\FV{\vp'}$ and
    $$
    \ha + \DNE{\Sigma_{k-1}} \,\, (\ha \text{ if }k=0)  \vdash \varphi \leftrightarrow \varphi' .
    $$
%    where $\DNE{\Sigma_{k-1}} $ is omitted if $k=0$.
\end{enumerate}
\end{theorem}
\begin{proof}
We mimic the proof of  Theorem \ref{thm: PNFT}. 
Thus we first prepare the following auxiliary assertion (which
%corresponds to the item \ref{item: PNFT(NEk. NNPk)} in the proof of Lemma \ref{PNFT_of_NE_k_and_NNU_k}):
is in fact a consequence from the item \ref{eq: item for Pi_k in PNFT for df-formulas}):
\begin{enumerate}
\setcounter{enumi}{2}
\item 
    \label{eq: item for NE_k in PNFT for df-formulas}
    if $\varphi \in \E_k^+$, then there exists $\varphi' \in \Pi_k$ such that $\FV{\vp}=\FV{\vp'}$ and
    $$
    \ha + \neg \neg \DNE{\Sigma_{k-1}}  \vdash \neg \varphi \leftrightarrow \neg \neg  \varphi'.
    $$
    \end{enumerate}
%As in the proof of Theorem \ref{thm: PNFT}, 
Then we show the items \ref{eq: item for Sigma_k in PNFT for df-formulas}, \ref{eq: item for Pi_k in PNFT for df-formulas} and 
\ref{eq: item for NE_k in PNFT for df-formulas} by induction on $k$ simultaneously.
The base case is trivial.
%(one can take $\vp'$ as $\vp$ itself).  
%The induction step is verified by induction on the structure of formulas which do not contain $\lor$.
Most of the parts for the induction step is the same as those for Theorem \ref{thm: PNFT}.
The same proof works since the logical principle in the items \ref{eq: item for Sigma_k in PNFT for df-formulas} and \ref{eq: item for Pi_k in PNFT for df-formulas} implies both of them for $k-1$ and  the logical principle in the item \ref{eq: item for NE_k in PNFT for df-formulas}
is the double negation of the logical principle in the item  \ref{eq: item for Pi_k in PNFT for df-formulas} as in Theorem \ref{thm: PNFT}.

Only the difference with the proof of Theorem \ref{thm: PNFT} is in proving the item \ref{eq: item for Sigma_k in PNFT} for $\vp : \equiv \vp_1 \to \vp_2\in \E_k^+$, where we use Lemma \ref{PNFT_of_NE_k_and_NNU_k}.
 % (this is only the part we use 
% in the proof of Theorem \ref{thm: PNFT}).
Here one can use the item \ref{eq: item for Pi_k in PNFT for df-formulas}
%{eq: item for NE_k in PNFT for df-formulas} (one of the induction hypothesis)
instead of Lemma \ref{PNFT_of_NE_k_and_NNU_k}.
This is because $\DNE{\Sigma_k}$ includes $\DNE{\Sigma_{k-1}}$ while the verification theory of the item \ref{eq: item for Sigma_k in PNFT} in Theorem \ref{thm: PNFT} contains  the verification theory of Lemma \ref{PNFT_of_NE_k_and_NNU_k}.
To be absolutely clear, we present the proof of this part:
Let $ \vp_1 \to \vp_2\in \E_k^+$.
By Lemma \ref{lem: basic facts on our classes}, we have $ \vp_1 \in \U_k^+$ and $\vp_2\in \E_k^+$.
By induction hypothesis,
%especially the item \ref{eq: item for Sigma_k in PNFT} for $\vp_2$, 
there exists 
$\rho_1(x_1)  \in \Sigma_{k-1}$ and 
$ \rho_2(x_2) \in \Pi_{k-1}$  such that
%\begin{itemize}
%\item
%$\ha +  \DNE{(\Pi_{k}\lor \Pi_{k})} \vdash \vp_1 \lr \forall x_1 \rho_1(x_1) $
%\item
$\ha + \DNE{\Sigma_{k-1}} \vdash \vp_1 \lr \forall x_1 \rho_1(x_1)$
and
$\ha + \DNE{\Sigma_{k}}  \vdash  \vp_2 \lr  \exists x_2 \rho_2(x_2) $.
By Lemma \ref{lem: basic facts on our classes}, we have $\neg \rho_2(x_2) \to  \neg \rho_1(x_1)$ is in $\U_{k-1}^+$.
Then,
%by hypothesis of course of values induction, 
by the item \ref{eq: item for Pi_k in PNFT} for $k-1$,
there exists $\xi(x_1, x_2) \in \Pi_{k-1}$ such that
$$
\ha +  \DNE{\Sigma_{k-2}} \vdash \xi(x_1, x_2) \lr  \left(\neg \rho_2(x_2) \to  \neg \rho_1(x_1) \right).$$
Then
%, using Lemma \ref{lem: k -> k-1}, we have that 
$\ha + \DNE{\Sigma_{k}}$ proves
$$
\begin{array}{cl}
& \vp_1 \to \vp_2 \\
%\lra&  \neg \neg \vp_1 \to \neg \neg  \vp_2\\
\underset{\text{[I.H.] }\DNE{\Sigma_{k}}}{\llr}& \vp_1 \to  \exists x_2 \rho_2(x_2)\\
\underset{\DNE{\Sigma_k}}{\llr}&  \vp_1 \to \neg \neg  \exists x_2 \rho_2(x_2)\\
%\llr& \neg \neg \vp_1 \to \neg \neg  \exists x_2 \rho_2(x_2)\\
\underset{\text{[I.H.] }\DNE{\Sigma_{k-1}}}{\llr}&  \forall x_1 \rho_1(x_1) \to \neg \neg  \exists x_2 \rho_2(x_2)\\
\llr&\neg \neg \exists x_2 \left( \forall x_1 \rho_1(x_1) \to \rho_2(x_2) \right) \\
\underset{\DNE{\Pi_{k-1}}}{\llr}&\neg \neg \exists x_2 \left(\neg  \rho_2(x_2) \to \neg \forall x_1 \rho_1(x_1)  \right) \\
\underset{\DNE{\Sigma_{k-1}}}{\llr}&\neg \neg \exists x_2 \left(\neg  \rho_2(x_2) \to \neg \neg \exists x_1  \neg  \rho_1(x_1)  \right) \\
\llr &\neg \neg \exists x_2 \neg \neg \exists x_1  \left(\neg  \rho_2(x_2) \to  \neg  \rho_1(x_1)  \right) \\
\underset{\text{[I.H.] }\DNE{\Sigma_{k-2}}}{\llr}&\neg \neg \exists x_1, x_2 \, \xi (x_1, x_2)\\
\underset{ \DNE{\Sigma_k}}{\llr}& \exists x_1, x_2 \, \xi (x_1, x_2) \in \Sigma_k .\\
\end{array}
$$
\end{proof}

\begin{remark}
It follows from Theorem \ref{thm: PNFT} and Corollary \ref{cor: pk v pk -DNE -> Ukp-DNE} that $\LEM{\E_k^+}$, $\LEM{\U_k^+}$ and $\DNE{\U_k^+}$ are equivalent to $\LEM{\Sigma_k}$, $\LEM{\Pi_k}$ and $\DNE{(\Pi_k \lor \Pi_k)}$ respectively  over $\ha$ (cf. \cite[Corollary 2.9]{ABHK04}).
This may not be the case for $\DNE{\E_k^+}$ and $\DNE{\Sigma_k}$.
On the other hand, Theorem \ref{thm: PNFT for df-formulas} implies that $\ha + \DNE{\Sigma_k}$ proves the double negation elimination for all formulas in $\E_k^+$ which do not contain $\lor.$
\end{remark}

\section{A conservation result}
\label{sec: conservativity}
In this section, we generalize a well-known fact that $\pa $ is $\Pi_2$-conservative over $\ha$ in the context of semi-classical arithmetic (see Theorem \ref{thm: conservation results for PNF -> Pi_k+2}).
The fact is normally shown by applying the negative translation followed by the Friedman A-translation (see e.g. \cite[Chapter 14]{Koh08}).
As for the negative translation, there are several equivalent forms (see \cite[Section 1.10.1]{Tro73}).
Here we employ Kuroda's negative translation among them.

\begin{definition}[cf. Definition 10.1 in \cite{Koh08}]
\label{def: NT}
Let $\vp$ be a $\ha$-formula.
Then its negative translation $\vp^N$ is defined as $\vp^N :\equiv \neg \neg  \SN{\vp}$, where $\SN{\vp}$ is defined by induction on the logical structure of $\vp$ as follows: 
\begin{itemize}
    \item 
    $\SNP{\vp_{\rm p}} :\equiv \vp_{\rm p}$ if $\vp_{\rm p}$ is a prime formula;
    \item
    $\SN{(\vp_1 \circ \vp_2)} :\equiv \SNP{\vp_1} \circ \SNP{\vp_2}$, where $\circ \in \{ \land, \lor, \to \}$;
    \item
    $\SN{(\exists x \vp_1)} :\equiv \exists x \SNP{\vp_1}$;
    \item
    $\SN{(\forall x \vp_1)} :\equiv \forall x \neg \neg \SNP{\vp_1}$.
\end{itemize}
\end{definition}

\begin{remark}
\label{rem: NT preserving FVs}
By induction on the structure of formulas, one can show $\FV{\vp} = \FV{\SN{\vp}} =\FV{\vp^N}$ for all formulas $\vp$.
When it is clear from the context, we suppress the argument on free variables.
\end{remark}

\begin{lemma}
\label{lem: vp -> vp*}
For any $\ha$-formula $\vp$ of the prenex normal form, $\ha $ proves $\vp \to \SN{\vp}$.
\end{lemma}
\begin{proof}
By induction on the structure of formulas of the prenex normal form.
\end{proof}

\begin{proposition}
\label{Soundness of NT}
For any $\ha$-formula $\vp$, if $\pa \vdash \vp$, then $\ha \vdash \vp^N$.
\end{proposition}
\begin{proof}
By induction on the length of the derivations (see the proof of \cite[Proposition 10.3]{Koh08}).
\end{proof}

\begin{lemma}
\label{lem: equivalence of vp and vp' over semi-classical arithmetic}
Let $k$ be a natural number.
\begin{enumerate}
    \item
    \label{item: vp'<-> vp in Sigma_k}
    For any $\ha$-formula $\vp\in \Sigma_k$, $\ha + \DNE{\Sigma_{k}}$ proves $\vp^N \lr \vp$.
    \item
    \label{item: vp'<-> vp in Pi_k}
     For any $\ha$-formula $\vp\in \Pi_k$, $\ha + \DNE{\Sigma_{k-1}}$ $(\ha$ if $k=0)$ proves $\vp^N \lr \vp$.
\end{enumerate}
\end{lemma}
\begin{proof}
By simultaneous induction on $k$.
The base case is trivial.
For the induction step, assume the items \ref{item: vp'<-> vp in Sigma_k} and \ref{item: vp'<-> vp in Pi_k} for $k$ to show those for $k+1$.
For the first item, let  $\exists x \vp_1 \in \Sigma_{k+1}$ where $\vp_1 \in \Pi_k$.
We have that $\ha + \DNE{\Sigma_{k+1}}$ proves
$$
\left( \exists x \vp_1 \right)^N\equiv \neg \neg \exists x \SNP{\vp_1} %\underset{\DNE{\Pi_{k}}}{\llr}
\lr
\neg \neg \exists x \neg \neg \SNP{\vp_1}
\underset{\text{[I.H.] }\DNE{\Sigma_{k-1}}}{\llr} \neg \neg \exists x \vp_1 \underset{\DNE{\Sigma_{k+1}}}{\llr}
\exists x \vp_1 .
$$
For the second item, let $\forall x \vp_1 \in \Pi_{k+1}$ where $\vp_1 \in \Sigma_k$.
Since $\DNE{\Pi_{k+1}}$ is derived from $\DNE{\Sigma_k}$ (see Lemma \ref{lem: k -> k-1}.\eqref{item: sk-DNE -> pk+1-DNE}), we have that $\ha + \DNE{\Sigma_k}$ proves
$$
\left( \forall x \vp_1 \right)^N\equiv \neg \neg \forall x \neg \neg \SNP{\vp_1} \underset{\text{[I.H.] }\DNE{\Sigma_k}}{\llr} \neg \neg \forall x \vp_1 \underset{\DNE{\Pi_{k+1}}}{\llr}
\forall x \vp_1 .
$$

\end{proof}

%In the context of the translation, without otherwise stated, wework in the language with an additional predicate symbol∗of arity 0, whichbehaves as a “place holder” (See [2,3] for more information).
Let $\ha^{\PH} $ denote $\ha$ in the extended language where a predicate symbol ${\PH} $ of arity $0$, which behaves as a “place holder”, is added.
In particular, $\ha^{\PH}$ has $\perp \to *$ as an axiom.
To make our arguments absolutely clear, we prefer to add the distinguished new predicate ${\PH} $
%(as done in \cite{Ishi00}) 
rather than discussing about A-translation inside the original language as done in \cite{Fri78, Koh08}.

\begin{definition}[A-translation \cite{Fri78}]
\label{def: A-translation}
%Let $\vp$ be a $\ha$-formula $\vp$.
%Then its A-translation $\vp^{\PH} $ is defined by induction on the logical structure of 
For a $\ha$-formula $\vp$, we define $\vp^{\PH} $ as a formula obtained from $\vp$ by replacing all the prime formulas $\vp_{\rm p}$ in $\vp$ with $\vp_{\rm p} \lor {\PH} $.
Officially, $\vp^{\PH} $ is defined by induction on the logical structure of $\vp$ as in Definition \ref{def: NT}.
In particular, $\perp^{\PH} :\equiv  \left(\perp \lor \, {\PH} \right)$, which is equivalent to  ${\PH} $ over $\ha^{\PH} $.
In what follows, $\neg_{\PH}\, \vp$ denotes $\vp \to {\PH} $.
\end{definition}

\begin{remark}
\label{rem: dollar-translation preserving FVs}
By induction on the structure of formulas, one can show $\FV{\vp} =\FV{\vp^{\PH}}$ for all $\ha$-formulas $\vp$.
\end{remark}

\begin{proposition}[Cf. Lemma 2 in \cite{Fri78}]
\label{prop: A-translation}
For any $\ha$-formula $\vp$, if $\ha \vdash \vp$, then $\ha^{\PH} \vdash \vp^{\PH} $.
\end{proposition}
\begin{proof}
By induction on the length of the derivations.
%(see \cite[Lemma 2]{Fri78} and \cite[Proposition 14.3]{Koh08}).
\end{proof}

\begin{remark}
An analogous assertion of Proposition \ref{prop: A-translation} holds for $\ha +\LEM{\Sigma_1}$ and $\ha^{\PH} +\LEM{\Sigma_1}$ instead of $\ha$ and $\ha^{\PH}$ respectively  (see \cite[Lemma 3.1]{KS14}).
\end{remark}

%For this purpose, we use 
The following substitution result is important in the application of the A-translation:
\begin{lemma}[Cf. Theorem 6.2.4 in \cite{vD13}]
\label{lem: Substitution}
Let $X$ be a set of $\ha$-sentences and $\vp$ be a $\ha^{\PH} $-formula.
If $\ha^{\PH} + X \vdash \vp$, then $\ha + X \vdash \vp[\psi/{\PH}]$ for any $\ha$-formula $\psi$ such that the free variables of $\psi$ are not bounded in $\vp$, where $\vp[\psi/{\PH}]$ is the $\ha$-formula obtained from $\vp$ by replacing all the occurrences of ${\PH}$ in $\vp$ with $\psi$.
\end{lemma}
\begin{proof}
Fix a set $X$ of $\ha$-sentences.
By induction on $k$,
%the length of the given proof of $\ha^*+X$, 
one can show straightforwardly that for any $k$ and any $\ha^{\PH} $-formula $\vp$, if $\ha^{\PH} + X \vdash \vp$ with the proof of length $k$, then $\ha + X \vdash \vp[\psi/{\PH} ]$ for any $\ha$-formula $\psi$ such that the free variables of $\psi$ is not bounded in $\vp$.
The variable condition is used to verify the case of axioms and rules for quantifiers.
\end{proof}

The following lemma is a key for our generalized conservation results.
\begin{lemma}
\label{lem: key lemma on vp*A in semi-classical arithemtic}
Let $k$ be a natural number.
\begin{enumerate}
    \item
    \label{item: vp*A <-> vp* v A for vp in Sigma_k}
    For any $\ha$-formula $\vp\in \Sigma_k$, $\ha^{\PH} + \LEM{\Sigma_{k-1}}$ $(\ha^{\PH} $ if $k=0)$ proves $\left( \SN{\vp}\right)^{\PH} \lr \SN{\vp} \lor {\PH} $.
    \item
        \label{item: vp*A <-> vp* v A for vp in Pi_k}
     For any $\ha$-formula $\vp\in \Pi_k$, $\ha^{\PH} + \LEM{\Sigma_{k}}$ proves $\left( \SN{\vp}\right)^{\PH} \lr \SN{\vp} \lor {\PH} $.
\end{enumerate}
\end{lemma}
\begin{proof}
We show the items \ref{item: vp*A <-> vp* v A for vp in Sigma_k} and \ref{item: vp*A <-> vp* v A for vp in Pi_k} simultaneously by induction on $k$.

The base case:
Since every quantifier-free formula $\QF{\vp}$ such that $\FV{\QF{\vp}}=\{\ol{x} \} $ is equivalent to a prime formula $t(\ol{x})=0$ for some closed term $t$ (see e.g. \cite[Proposition 3.8]{Koh08}), by Proposition \ref{prop: A-translation}, it suffices to show the assertions only for prime formulas.
Since
$\left(\SNP{\vp_{\rm p}}\right)^{\PH} \equiv {\vp_{\rm p}}^{\PH} \equiv \vp_{\rm p} \lor {\PH} \equiv \SNP{\vp_{\rm p}} \lor {\PH}$,
we are done.

The induction step:
Assume that the items \ref{item: vp*A <-> vp* v A for vp in Sigma_k} and \ref{item: vp*A <-> vp* v A for vp in Pi_k} hold for $k$.
We first show the item 1 for $k+1$.
Let $\vp_1\in \Pi_k$.
Since
$$\left(\SNP{\exists x \vp_1}\right)^{\PH} \equiv \left( \exists x\, \SNP{\vp_1}\right)^{\PH} \equiv \exists x \left(\SNP{\vp_1}\right)^{\PH}, $$
by induction hypothesis, we have
$$
\ha^{\PH} + \LEM{\Sigma_k} \vdash\left(\SNP{\exists x \vp_1}\right)^{\PH} \lr \exists x \left(\SNP{\vp_1} \lor {\PH} \right).$$
Since $\ha^{\PH} $ proves $\exists x \left(\SNP{\vp_1} \lor {\PH} \right) \lr \left(\exists x \, \SNP{\vp_1} \lor {\PH} \right) \equiv \left( \SNP{\exists x \vp_1 } \lor {\PH} \right)$,
we have
$$
\ha^{\PH} + \LEM{\Sigma_k} \vdash \left(\SNP{\exists x \vp_1}\right)^{\PH} \lr  \left( \SNP{ \exists x \vp_1} \lor {\PH} \right).$$
Thus we have shown the item \ref{item: vp*A <-> vp* v A for vp in Sigma_k} for $k+1$.

Next, we show the item 2 for $k+1$.
Let $\vp_2\in \Sigma_k$.
We shall show that $\ha^{\PH} + \DNE{\Sigma_k}$ (and hence, $\ha^{\PH} + \LEM{\Sigma_{k+1}}$) proves
$\SNP{\forall x \vp_2} \lor {\PH} \to
\left( \SNP{ \forall x \vp_2} \right)^{\PH} $.
By Lemma \ref{lem: equivalence of vp and vp' over semi-classical arithmetic}.\eqref{item: vp'<-> vp in Sigma_k}, we have
\begin{equation}
\label{eq: vp2 <-> NNvp_2*}
\ha+\DNE{\Sigma_k} \vdash \vp_2 \lr \left(\vp_2 \right)^N \equiv \neg \neg \SNP{\vp_2}.
\end{equation}
%Since ${\vp_2}^N \equiv \neg \neg \vp_2^*$, by Lemm \ref{lem: equivalence of vp and vp^N over semi-classical arithmetic}, 
Then we have that
$\ha^{\PH} + \DNE{\Sigma_k}$ proves
$$
\SNP{\forall x \vp_2}\lor {\PH} \equiv  \left(  \forall x \neg \neg \SNP{\vp_2}  \lor {\PH} \right) \lr \forall x \vp_2 \lor {\PH} .
$$
%Since $\vp_2$ is of the prenex normal form
By Lemma \ref{lem: vp -> vp*}, $\ha$ proves $\vp_2 \to \SNP{\vp_2}$.
Then, using induction hypothesis and the fact that $\DNE{\Sigma_{k}}$ derives $\LEM{\Sigma_{k-1}}$, we have that $\ha^{\PH} + \DNE{\Sigma_k}$ proves
$$
\begin{array}{rcl}
\SNP{\forall x \vp_2} \lor {\PH}
%\equiv  \forall x \neg \neg {\vp_2}^* \lor {\PH}
&\underset{\DNE{\Sigma_k}}{\llr}& \forall x \vp_2 \lor {\PH}\\
&\lra &  \forall x \SNP{\vp_2} \lor {\PH}\\[2pt]
& \lra & \forall x (\SNP{\vp_2} \lor {\PH})  \\[2pt]
& \underset{\text{[I.H.] }\LEM{\Sigma_{k-1}}}{\llr} &\forall x \left(\SNP{\vp_2} \right)^{\PH}\\
& \lra & \forall x \left(\left( \left(\SNP{\vp_2} \right)^{\PH} \to {\PH} \right) \to {\PH} \right)\\
&\llr & \left( \SNP{\forall x \vp_2}\right)^{\PH} .
\end{array}
$$
In the following, we show the converse direction:
\begin{equation}
    \label{eq: the direction -> of item 2 for k+1}
\ha^{\PH} + \LEM{\Sigma_{k+1}} \vdash \left( \SNP{\forall x \vp_2} \right)^{\PH}  \to \SNP{\forall x \vp_2} \lor {\PH} .
\end{equation}
Reason in $\ha^{\PH} + \LEM{\Sigma_{k+1}} $.
Suppose $\left( \SNP{\forall x \vp_2 } \right)^{\PH} $, equivalently, \begin{equation}
\label{eq: forall x vp2 * dollar}
\forall x \left(\left( \left(\SNP{\vp_2} \right)^{\PH} \to {\PH} \right) \to {\PH} \right).
\end{equation}
By induction hypothesis,
%we have that
%$\ha^{\PH} + \LEM{\Sigma_{k}}$ proves that
$\eqref{eq: forall x vp2 * dollar} $ is equivalent to  $\forall x \left(\left(\SNP{\vp_2} \lor {\PH}  \to {\PH} \right) \to {\PH} \right)$, which is intuitionistically equivalent to
\begin{equation*}
        \label{eq: forall x dd vp2 *}
\forall x \left(\left(\SNP{\vp_2} \to {\PH} \right)\to {\PH} \right) .
\end{equation*}
Then we have
\begin{equation}
    \label{eq: exists x N vp2* -> d}
\exists x \neg \SNP{\vp_2} \to {\PH}  .
\end{equation}
By Lemma \ref{lem: NP and NS}.\eqref{item: NS}, there exists $\psi_2 \in \Pi_k$ such that $\FV{\vp_2}=\FV{\psi_2}$ and $\neg \vp_2$ is equivalent to $\psi_2$.
Since $\exists x \psi_2\in \Sigma_{k+1}$, by $\LEM{\Sigma_{k+1}}$, we have %$\ha + \LEM{\Sigma_{k+1}}$ proves 
$\exists x \psi_2 \lor \neg \exists x \psi_2$, and hence, 
\begin{equation*}
\exists x \neg \vp_2 \lor \forall x \neg \neg \vp_2 .
\end{equation*}
Then, by \eqref{eq: vp2 <-> NNvp_2*}, we obtain
\begin{equation*}
\exists x \neg \SNP{\vp_2} \lor \forall x \neg \neg \SNP{\vp_2} .
\end{equation*}
In the former case, we have ${\PH}$ by \eqref{eq: exists x N vp2* -> d}.
In the latter case, we have $\SNP{\forall x {\vp_2}}$.
Thus we have shown \eqref{eq: the direction -> of item 2 for k+1}.
\end{proof}

\begin{lemma}
\label{lem: basic facts on *}
%[Cf. Lemma 2 in \cite{Ishi00}]
Let $\vp$ be a $\ha^{\PH} $-formula.
\begin{enumerate}
    \item 
    \label{item: vp -> ddvp}
$\ha^{\PH} \vdash \vp \to \neg_{\PH} \neg_{\PH} \vp$.
\item
\label{item: Advp <-> dEvp}
$\ha^{\PH} \vdash \forall x \neg_{\PH} \vp \lr \neg_{\PH} \exists x \vp$.
\item
\label{item: dddvp -> dvp}
$\ha^{\PH} \vdash  \neg_{\PH} \neg_{\PH} \neg_{\PH} \vp \to \neg_{\PH} \vp$.
    \item
    \label{item: Eddvp -> ddEvp}
    $\ha^{\PH} \vdash \exists x \neg_{\PH} \neg_{\PH} \vp \to \neg_{\PH} \neg_{\PH} \exists  x \vp$.
\end{enumerate}
\end{lemma}
\begin{proof}
\eqref{item: vp -> ddvp}, \eqref{item: Advp <-> dEvp} and \eqref{item: dddvp -> dvp} are immediate from the definition of $\neg_{\PH} $ (see Definition \ref{def: A-translation}).
\eqref{item: Eddvp -> ddEvp} follows from \eqref{item: vp -> ddvp}, \eqref{item: Advp <-> dEvp} and \eqref{item: dddvp -> dvp}.
\end{proof}

\begin{lemma}
\label{lem: vp -> vp'A for any vp:PNF}
For any $\ha$-formula $\vp$ of the prenex normal form, $\ha^{\PH} \vdash \vp \to \left(\vp^N \right)^{\PH} $.
\end{lemma}
\begin{proof}
Since there exists a closed term $t$ such that $\ha \vdash \QF{\vp}(x_1, \dots, x_k) \lr t(x_1, \dots, x_k) =0 $ for each quantifier-free formula $\QF{\vp}$ such that $\FV{\QF{\vp}}=\{x_1, \dots, x_k \}$ (see e.g. \cite[Proposition 3.8]{Koh08}), by Proposition \ref{Soundness of NT} and Proposition \ref{prop: A-translation},
one can assume that formulas of the prenex normal form consist of the formulas of form $Q_{1}x_1 \dots Q_{k}x_k \, \vp_{\rm p} $ where $Q_i$s are quantifiers and $\vp_{\rm p}$ is prime.
We show our assertion by induction on the structure of formulas of this form.

For a prime formula $\vp_{\rm p}$, it is trivial to see that $\ha^{\PH} $ proves
$$
\vp_{\rm p} \to \vp_{\rm p} \lor {\PH} \to  \neg_{\PH} \neg_{\PH} \left(  \vp_{\rm p} \lor {\PH}  \right) \lr  \left(\left(\vp_{\rm p}\right)^N \right)^{\PH}.
%\vp_{\rm p} \to \neg \neg \vp_{\rm p} \to \to \neg \neg \left(\vp_{\rm p} \lor {\PH} \right) \equiv   \left({\vp_{\rm p}}^N \right)^{\PH}.
$$

Assume the assertion for $\vp$.
Then, using Lemma \ref{lem: basic facts on *},
$\ha^{\PH}$ proves
$$
\begin{array}{r}
\exists x \vp 
\underset{\text{[I.H.]}}{\lra}
\exists x \left(\vp^N \right)^{\PH} 
\equiv
\exists x \left(\neg \neg \SN{\vp} \right)^{\PH}
\lr
\exists x \neg_{\PH} \neg_{\PH} \left(\SN{\vp} \right)^{\PH} 
\to
\neg_{\PH} \neg_{\PH} \exists x  \left(\SN{\vp} \right)^{\PH} \\
\lr
\left(\left( \exists x \vp \right)^N \right)^{\PH}
\end{array}
$$
and
$$
\begin{array}{r}
\forall x \vp 
\underset{\text{[I.H.]}}{\lra}
\forall x \left(\vp^N \right)^{\PH} 
\equiv
\forall x \left(\neg \neg \SN{\vp} \right)^{\PH}
\lr
\forall x \neg_{\PH} \neg_{\PH} \left(\SN{\vp} \right)^{\PH} 
\to \neg_{\PH} \neg_{\PH} \forall x \neg_{\PH} \neg_{\PH} \left(\SN{\vp} \right)^{\PH} \\
\lr
\left(\left( \forall x \vp \right)^N \right)^{\PH} .
\end{array}
$$
\end{proof}

\begin{theorem}
\label{thm: conservation results for PNF -> Pi_k+2}
Let $k$ be a natural number.
For any formulas $\vp \in \Pi_{k+2}$ and $\psi$ of the prenex normal form, if $\pa \vdash \psi \to \vp$, then $\ha + \LEM{\Sigma_k} \vdash \psi \to \vp$.
\end{theorem}
\begin{proof}
Let $\vp :\equiv \forall x \exists y \vp_1$ where $\vp_1\in \Pi_k$.  
Since one can freely replace the bound variables, assume that the free variables of $ \exists y \vp_1 $ are not bounded in $\psi$ and $x$ does not occur in $\psi$ without loss of generality.

Suppose $\pa \vdash \psi \to \forall x \exists y \vp_1$.
%By the negative translation (see e.g. \cite[Theorem 1.10.11]{Tro73}),
By Proposition \ref{Soundness of NT}, we have that
$\ha$ proves $\neg \neg (\SN{\psi} \to \forall x \neg \neg \exists y\, \SNP{\vp_1})$, which is intuitionistically equivalent to $\neg \neg \SN{\psi} \to  \forall x  \neg \neg \exists y\, \SNP{\vp_1}$, namely, $\psi^N \to  \forall x  \neg \neg \exists y \, \SNP{\vp_1}$.
Then we have
$$
\ha \vdash \psi^N \to   \neg \neg \exists y \, \SNP{\vp_1} .
$$
By Proposition \ref{prop: A-translation}, we have
$$
\ha^{\PH} \vdash \left(\psi^N\right)^{\PH} \to   \neg_{\PH} \neg_{\PH} \exists y \left( \SNP{\vp_1}\right)^{\PH} ,
$$
and hence,
$$
\ha^{\PH} \vdash \psi \to  \neg_{\PH} \neg_{\PH} \exists y \left( \SNP{\vp_1}\right)^{\PH}
$$
by Lemma \ref{lem: vp -> vp'A for any vp:PNF}.
Then, by Lemma \ref{lem: key lemma on vp*A in semi-classical arithemtic}.\eqref{item: vp*A <-> vp* v A for vp in Pi_k}, we have that
$\ha^{\PH} +\LEM{\Sigma_k} $ proves
$$\psi \to  \neg_{\PH} \neg_{\PH} \exists y \left( \SNP{\vp_1} \lor {\PH} \right),$$
 which is intuitionistically equivalent to 
 $$\psi \to  \neg_{\PH} \neg_{\PH} \exists y \, \SNP{\vp_1} .$$
Since the free variables of $ \exists y \vp_1 $ are not bounded in $\psi$, using Lemma \ref{lem: Substitution} with Remark \ref{rem: NT preserving FVs}, we have
 \begin{equation}
\label{eq: cons afrer applying A-translation}
\ha +\LEM{\Sigma_k} \vdash \psi \to \left( (\exists y \, \SNP{\vp_1} \to \exists y \vp_1) \to \exists y \vp_1\right).
 \end{equation}
On the other hand, by Lemma \ref{lem: equivalence of vp and vp' over semi-classical arithmetic}.\eqref{item: vp'<-> vp in Pi_k} and the fact that $ \LEM{\Sigma_k}$ derives $\DNE{\Sigma_k}$, we have that $\ha +\LEM{\Sigma_k}$ proves
\begin{equation*}
\label{eq: HA + Sk-1-DNE |- vp* -> vp}
 \SNP{\vp_1} \to \neg \neg \SNP{\vp_1} \equiv \left(\vp_1\right)^N \underset{ \DNE{\Sigma_{k-1}}}{\llr}  \vp_1.
\end{equation*}
and hence, $\exists y \, \SNP{\vp_1} \to \exists y \vp_1$.
Then, by \eqref{eq: cons afrer applying A-translation}, we have $\ha +\LEM{\Sigma_k} \vdash \psi \to \exists y \vp_1$.
By our assumption, $x$ does not occur in $\psi$, and hence, $\ha +\LEM{\Sigma_k} \vdash \psi \to \forall x \exists y \vp_1 $ follows.
\end{proof}

%\begin{remark}
We have shown Theorem \ref{thm: conservation results for PNF -> Pi_k+2} in order to prove the optimality of our prenex normal form theorems in Section \ref{sec: PNFT} (see Section \ref{sec: Optimality}).
On the other hand, the conservation result on semi-classical arithmetic itself is interesting.
This will be studied comprehensively in \cite{FK20-2}.
%In fact, Theorem \ref{thm: conservation results for PNF -> Pi_k+2} is a generalization of the well-known result that $\pa$ is $\Pi_2$-conservative over $\ha$, which is usually shown by using the negative translation followed by A-translation.
%\end{remark}

%The following proposition asserts that Theorem \ref{thm: conservation results for PNF -> Pi_k+2} is optimal:

%\begin{theorem}
%\label{thm: conservativity}
%For each formula $\vp \in \Pi_{k+2}$, if $\pa \vdash \vp$, then $\ha + \LEM{\Sigma_k} \vdash \vp$.
%\end{theorem}
%\begin{proof}
%\TBR
%\end{proof}

\section{Characterizations}\label{sec: Optimality}
%From the perspective of Remark \ref{rem: On the verification theory}, 
\begin{notation}
Let $\T $ be an extension of $\ha$.
%\subseteq \pa$.
Let $\Gamma$ and $\Gamma'$ be classes of $\ha$-formulas.
%\begin{itemize}
 %   \item 
  Then  $\PNFT{\Gamma}{\Gamma'} $ denotes the following statement:
%  \begin{center}
  for any $\varphi \in \Gamma$, there exists $\varphi' \in \Gamma'$ such that $\FV{\vp}=\FV{\vp'}$ and
    $\T  \vdash \varphi \leftrightarrow \varphi' $.
%\end{center}
%\end{itemize}
\end{notation}

Under this notation, Theorem \ref{thm: PNFT} asserts (modulo Remark \ref{rem: sigma_k and pi_k}) that
%\begin{enumerate}
 %   \item 
for a semi-classical theory $T$ containing 
 %$\ha +\DNE{\Sigma_k}+\DNS{\U_k^+}$,
%(resp. 
$\ha +  \DNE{(\Pi_k\lor \Pi_k)} $, 
%$\PNFT{\E_{k'}}{\Sigma_{k'}}$
%(resp. 
$\PNFT{\U_{k'}}{\Pi_{k'}}$ 
holds for all $k' \leq k$ as well as the analogous assertion for $\E_k$ and $\Sigma_k$.
%$\ha +  \DNE{(\Pi_k\lor \Pi_k)} $ and $\PNFT{\U_{k'}}{\Pi_{k'}}$. 
%\end{enumerate}
It is natural to ask whether the verification theories
%of our prenex normal form theorem (Theorem \ref{thm: PNFT}) 
are optimal.
%(cf. Remark \ref{rem: On the verification theory}).
In this section, among other things (see Table \ref{table: PNFT}), we show that this is exactly the case:
%For this purpose, we introduce the following notation.
%In what follows, among other things, we show the following:
\begin{enumerate}
    \item 
    For a theory $\T$ in-between $\ha $ and $\pa$,
%n extension of $\ha$ and  subtheory of $\pa$.
 $\T \vdash \DNE{(\Pi_k\lor \Pi_k)}$ if and only if $\PNFT{\U_{k'}}{\Pi_{k'}}$ for all $k' \leq k$. (Theorem \ref{thm: characterization of PNFT(Ukp,Pk)})
\item
 For a theory $\T$ in-between $\ha +\LEM{\Pi_{k-1}} $ ($\ha$ if $k=0$) and $\pa$,
%n extension of $\ha$ and  subtheory of $\pa$.
 $\T \vdash \DNE{\Sigma_k} + \DNS{\U_k^+} $ if and only if $\PNFT{\E_{k'}}{\Sigma_{k'}}$ for all $k' \leq k$.
 (Theorem \ref{thm: characterization of PNFT(Ek,Sk)})
\end{enumerate}
%These assertions show that the verification theories of the items \ref{eq: item for Sigma_k in PNFT} and \ref{eq: item for Pi_k in PNFT} in Theorem \ref{thm: PNFT} are optimal.
%In this section, we show the optimality of our prenex normal form theorem.

%For this purpose, we first show some lemmata.

\begin{lemma}
\label{lem: PNFT(Uk, Ok) + Sk-1-LEM => PkPk-DNE}
Let $\T $ be a theory in-between $\ha +\LEM{\Sigma_{k-2}}\, (\ha$ if $k<2)$ and $\pa$.
If $\PNFT{\U_{k}}{\Pi_{k}}$, then $\T \vdash \DNE{(\Pi_{k}\lor \Pi_{k})}$.
\end{lemma}
\begin{proof}
Fix an instance of $\DNE{(\Pi_{k}\lor \Pi_{k})}$
$$
\vp :\equiv \forall x \left(\neg\neg \left(\vp_1(x) \lor \vp_2(x) \right) \to \vp_1(x) \lor \vp_2(x)  \right),
$$
where $\vp_1(x), \vp_2(x) \in \Pi_{k}(x)$.
Since $\neg\neg \left(\vp_1(x) \lor \vp_2(x) \right) $ and $ \vp_1(x) \lor \vp_2(x)$ are in $\U_{k}$, by our assumption, there exist $\rho(x)$ and $\rho'(x)$ in $\Pi_{k}(x)$ such that
%$\FV{ \rho(x)} = \FV{ \rho'(x)} = \FV{\vp_1 (x) \lor \vp_2 (x) } = \FV{ \neg\neg (\vp_1 (x) \lor \vp_2 (x))}$ and
$\T $ proves $ \rho(x) \lr \vp_1(x) \lor \vp_2(x) $  and $\rho'(x) \lr \neg \neg \left(\vp_1(x) \lor \vp_2(x)\right) $.
%Put $\vp' :\equiv \forall x \left(\rho'(x) \to \rho(x) \right)$.
Since $\pa \vdash \vp$ and $\pa$ is an extension of $\T$, we have $\pa \vdash  \rho'(x) \to \rho(x) $.
By Theorem \ref{thm: conservation results for PNF -> Pi_k+2}, we have that $\ha +\LEM{\Sigma_{k-2}}$ proves $\rho'(x) \to \rho(x)$, and hence, $\forall x\left(
\rho'(x) \to \rho(x) \right)$.
Since $\T$ is an extension of $\ha + \LEM{\Sigma_{k-2}}$, we have $\T \vdash \forall x\left(\rho'(x) \to \rho(x) \right)$, and hence, $\T \vdash \vp$.
\end{proof}

\begin{lemma}
\label{lem: PNFT(Uk, Pk)=>Sk-1-LEM}
Let $\T $ be a theory in-between $\ha $ and $\pa$.
If $\PNFT{\U_{k'}}{\Pi_{k'}}$ for all $k' \leq k$, then $\T \vdash \LEM{\Sigma_{k-1}}$.
\end{lemma}
\begin{proof}
By induction on $k$.
The base case is trivial.
For the induction step, assume $\PNFT{\U_{k'}}{\Pi_{k'}}$ for all $k' \leq k+1$.
Then, by induction hypothesis, we have  $\T \vdash \LEM{\Sigma_{k-1}}$.
Fix an instance of $\LEM{\Sigma_{k}}$
$$
\vp :\equiv \forall x (\vp_1(x) \lor \neg \vp_1(x)),
$$
where $\vp_1(x) \in \Sigma_{k}(x)$.
Since $\vp \in \U_{k+2}$, by our assumption, there exists a sentence $\vp' \in \Pi_{k+1}$ such that
%$\FV{\vp} = \FV{\vp'}$ and 
$\T \vdash \vp \lr \vp'$.
Since $\pa \vdash \vp$, we have $\pa \vdash \vp'$.
Then, by Theorem \ref{thm: conservation results for PNF -> Pi_k+2}, we have $\ha + \LEM{\Sigma_{k-1}} \vdash \vp'$, and hence, $\T \vdash \vp'$.
Thus we have $\T \vdash \vp$.
\end{proof}

\begin{theorem}
\label{thm: characterization of PNFT(Ukp,Pk)}
Let $\T $ be a theory in-between $\ha $ and $\pa$.
%n extension of $\ha$ and  subtheory of $\pa$.
Then $\T \vdash \DNE{(\Pi_k\lor \Pi_k)}$ if and only if $\PNFT{\U_{k'}}{\Pi_{k'}}$ for all $k' \leq k$.
\end{theorem}
\begin{proof}
The ``only if'' direction is immediate from Theorem \ref{thm: PNFT}.\eqref{eq: item for Pi_k in PNFT} and Lemma \ref{lem: sigma_k and pi_k}.
We show the converse direction.
Assume $\PNFT{\U_{k'}}{\Pi_{k'}}$ for all $k' \leq k$.
Let $k>0$ without loss of generality. 
By Lemma \ref{lem: PNFT(Uk, Pk)=>Sk-1-LEM}, $\T \vdash \LEM{\Sigma_{k-1}}$.
Then, by Lemma \ref{lem: PNFT(Uk, Ok) + Sk-1-LEM => PkPk-DNE}, we have $\T \vdash \DNE{(\Pi_k\lor \Pi_k)} $.
\end{proof}

\begin{definition}
Let $\Gamma$ be a class of formulas.
Then $\dn{\Gamma}$ denotes the class of $\ha$-formulas $\neg \neg \vp$ where $\vp \in \Gamma$, and  $\n{\Gamma}$ denotes that for $\neg \vp$ where $\vp \in \Gamma$.
\end{definition}

\begin{lemma}
\label{lem: PNFT(NNUk, NNPk)=>NNSk-1-LEM}
Let $\T $ be a theory in-between $\ha $ and $\pa$.
If $\PNFT{\dn{\U_{k'}}}{\dn{\Pi_{k'}}}$ for all $k' \leq k$, then $\T \vdash \NN{\LEM{\Sigma_{k-1}}}$.
\end{lemma}
\begin{proof}
By induction on $k$.
The base case is trivial.
For the induction step, assume $\PNFT{\dn{\U_{k'}}}{\dn{\Pi_{k'}}}$ for all $k' \leq k+1$.
Then, by induction hypothesis, $\T$ proves $\NN{\LEM{\Sigma_{k-1}}}$.
Fix an instance of $\NN{\LEM{\Sigma_{k}}}$
$$
\vp :\equiv \neg \neg \forall x (\vp_1(x) \lor \neg \vp_1(x)),
$$
where $\vp_1(x) \in \Sigma_{k}(x)$.
Since $\vp \in \dn{\U_{k+1}}$, by our assumption, there exists a sentence $\vp' \in \Pi_{k+1}$ such that
%$\FV{\vp} = \FV{\vp'}$ and 
$\T \vdash \vp \lr \neg \neg \vp'$.
Since $\pa \vdash \forall x (\vp_1(x) \lor \neg \vp_1(x))$, we have $\pa \vdash \vp'$.
Then, by Theorem \ref{thm: conservation results for PNF -> Pi_k+2}, we have $\ha + \LEM{\Sigma_{k-1}} \vdash \vp'$, and hence, $\ha + \NN{\LEM{\Sigma_{k-1}}} \vdash \neg\neg \vp'$ by Lemma \ref{lem: HA+P |- vp => HA+NNP |- NNvp}.
Then $\T \vdash \vp$.
\end{proof}

\begin{theorem}
\label{thm: characterization of PNFT(NNUkp,NNPk)}
Let $\T $ be a theory in-between $\ha $ and $\pa$.
%n extension of $\ha$ and  subtheory of $\pa$.
The following are pairwise equivalent:
\begin{enumerate}
    \item 
    \label{item: PNFT(NEk. NNPk)}
    $\PNFT{\n{\E_{k'}}}{ \dn{\Pi_{k'}}}$ for all $k' \leq k$;
    \item
        \label{item: PNFT(NNUk. NNPk)}
    $\PNFT{\dn{\U_{k'}}}{ \dn{\Pi_{k'}}}$ for all $k' \leq k$;
  %  \item
%        \label{item: T|-NNPkPk-DNE}
 %   $\T \vdash \NN{\DNE{(\Pi_k\lor \Pi_k)}}$;
    \item
    \label{item: T|-Ukp-DNS}
    $\T \vdash  \DNS{\U_k^+}$.
\end{enumerate}
%Then $\T \vdash \NN{\DNE{(\Pi_k\lor \Pi_k)}}$ if and only if $\PNFT{\dn{\U_{k'}}}{ \dn{\Pi_{k'}}}$ for all $k' \leq k$.
\end{theorem}
\begin{proof}
The equivalence of \eqref{item: PNFT(NEk. NNPk)} and \eqref{item: PNFT(NNUk. NNPk)} is trivial (cf. the proof of Lemma \ref{PNFT_of_NE_k_and_NNU_k}).
%\ref{rem: eq on NE_k and NNU_k}). 
In addition, $(\ref{item: T|-Ukp-DNS} \to \ref{item: PNFT(NNUk. NNPk)})$ is immediate from Lemma \ref{PNFT_of_NE_k_and_NNU_k}.
%and Lemma \ref{lem: sigma_k and pi_k}.
In what follows, we show $(\ref{item: PNFT(NNUk. NNPk)} \to \ref{item: T|-Ukp-DNS})$.
Assume $\PNFT{\dn{\U_{k'}}}{\dn{\Pi_{k'}}}$ for all $k' \leq k$.
Since $\DNS{\U_{k}^+}$ is intuitionistically equivalent to $\NN{\DNS{\U_{k}^+}}$ (See Remark \ref{rem: NNDNS <-> DNS}), it suffices to show $\T \vdash \NN{\DNS{\U_{k}^+}}$.
Let $k>0$ without loss of generality. 
Fix an instance of $\NN{\DNS{\U_k^+}}$
$$\vp :\equiv \neg \neg \forall x \left( \forall y \neg \neg \vp_1(x,y) \to  \neg \neg \forall y  \vp_1(x,y)  \right)$$
where $\vp_1(x,y)\in \U_k^+(x,y)$.
By  Lemma \ref{lem: basic facts on our classes}, we have that $\forall y\neg \neg \vp_1(x,y)$ and  $\forall y\vp_1(x,y)$ are in $\U_k^+(x)$.
%By the definition of alternation paths (see 
Since $i(s) \equiv -$ for all alternation paths $s$ of $\forall y\neg \neg \vp_1(x,y)$ and $\forall y \vp_1(x,y)$, it is straightforward to show that there exists $k'\leq k$ such that $\forall y\neg \neg \vp_1(x,y)$ and  $\forall y\vp_1(x,y)$ are in $\U_{k'}(x)$.
Then, by $\PNFT{\dn{\U_{k'}}}{ \dn{\Pi_{k'}}}$,
there exist $\rho(x), \rho'(x)\in \Pi_{k'}(x)$
such that
$\T $ proves  $\neg \neg \rho(x) \lr \neg \neg \forall y \vp_1(x,y) $ and  $\neg \neg \rho'(x) \lr \neg \neg \forall y \neg \neg \vp_1(x,y) $.
%Let $\psi\in \Pi_{k+2}$ be a sentence of the prenex normal form which is obtained from 
%Let
%$$\psi :\equiv \forall x \left( \rho(x) \to \rho'(x) \right).$$
Since $\pa$ is an extension of $\T$ and $\pa \vdash \vp$, we have $\pa \vdash \rho'(x) \to \rho(x)$.
Then, by Theorem \ref{thm: conservation results for PNF -> Pi_k+2}, we have
that $\ha + \LEM{\Sigma_{k'-2}} \, (\ha$ if $k'<2)$ proves $\rho'(x) \to \rho(x)$, and hence, $\forall x \left( \neg \neg \rho'(x) \to \neg \neg \rho(x) \right)$.
%\forall x \left( \forall y \rho'(x,y) \to \rho(x) \right)$ as in the proof of Proposition \ref{prop: Ukp-DNS <-> NNPkPk-DNE}.
By Lemma \ref{lem: HA+P |- vp => HA+NNP |- NNvp}, we have
%$\ha + \NN{\LEM{\Sigma_{k'-2}}} \vdash \neg \neg \psi$, and hence,
\begin{equation*}
\label{eq: HA+NNSk'-2 |- psi}
\ha + \NN{\LEM{\Sigma_{k'-2}}} \vdash \neg \neg \forall x \left( \neg \neg \rho'(x) \to \neg \neg \rho(x) \right) .   
\end{equation*}
%and hence, 
%$\ha + \NN{\LEM{\Sigma_k}} \vdash \neg \neg \forall x \left(\forall y \rho'(x,y) \to \rho(x)\right)$.
On the other hand,
%by Lemma \ref{lem: k -> k-1}.\eqref{item: sk-DNE -> pk+1-DNE}, we have that $\ha +\DNE{\Sigma_{k'-1}} \, (\ha$ if $k'<1)$ proves $ \forall x,y (\rho'(x,y) \lr \neg \neg \rho'(x,y))$,
%Then, by Lemma \ref{lem: HA+P |- vp => HA+NNP |- NNvp}, we have
%\begin{equation}
%    \label{eq: HA+NNDNESk'-1 |- NN...}
%\ha +\NN{\DNE{\Sigma_{k'-1}}} \vdash \neg \neg \forall x,y (\rho'(x,y) \lr \neg \neg \rho'(x,y)).
%\end{equation}
%In addition, 
by Lemma \ref{lem: PNFT(NNUk, NNPk)=>NNSk-1-LEM} and our assumption, we have
%\begin{equation*}
 %   \label{eq: T |- NNLEMSk-1}
$\T \vdash \NN{\LEM{\Sigma_{k-1}}}$.
%\end{equation*}
%Since $\NN{\LEM{\Sigma_{k-1}}}$ implies $\NN{\LEM{\Sigma_{k'-2}}}$ and $\NN{\DNE{\Sigma_{k'-1}}}$, 
%By \eqref{eq: HA+NNSk'-2 |- psi}, \eqref{eq: HA+NNDNESk'-1 |- NN...} and \eqref{eq: T |- NNLEMSk-1},
Then we have
$$
\T \vdash \neg \neg \forall x \left( \neg \neg \forall y \neg \neg \vp_1(x,y) \to \neg \neg \forall y \vp_1(x,y)
\right),
$$
and hence, $\T \vdash \vp$.
%Thus we have shown $\T \vdash \DNS{\U_k^+}$.
\end{proof}

\begin{remark}
Theorem \ref{thm: characterization of PNFT(NNUkp,NNPk)} shows that the verification theory for Lemma \ref{PNFT_of_NE_k_and_NNU_k} is optimal.
\end{remark}

\begin{definition}\label{def: Gamma-v}
Let $\Gamma$ be a class of $\ha$-formulas.
$\df{\Gamma}$ denotes the class of formulas in $\Gamma$ which do not contain $\lor$.
\end{definition}

\begin{lemma}
\label{lem: PNFT(Ekm, Sk) + Pk-1-LEM => Sk-DNE}
Let $\T $ be a theory in-between $\ha +\LEM{\Pi_{k-1}} $ $(\ha$ if $k=0)$ and $\pa$.
If $\PNFT{\df{\E_{k'}}}{\Sigma_{k'}}$ for all $k' \leq k$, then $\T \vdash \DNE{\Sigma_{k}}$.
\end{lemma}
\begin{proof}
By induction on $k$.
The base case is trivial.
For the induction step, assume the assertion for $k$ and let $\T$ be a theory in-between $\ha +\LEM{\Pi_{k}}$ and $\pa$.
Assume also that $\PNFT{\df{\E_{k'}}}{\Sigma_{k'}}$ holds for all $k' \leq k+1$.
Then, by induction hypothesis, $\T$ proves $\DNE{\Sigma_{k}}$.
Since $\T$ contains $\ha +\LEM{\Pi_{k}},$ we have $\T \vdash \LEM{\Sigma_{k}}$ by \cite[Theorem 3.1(ii)]{ABHK04}.
Fix an instance of $\DNE{\Sigma_{k+1}}$
$$
\vp :\equiv \forall x ( \neg \neg \vp_1(x) \to \vp_1(x)),
$$
where $\vp_1(x) \in \Sigma_{k+1}(x)$.
Without loss of generality, one can assume that $\vp_1(x)$ does not contain $\lor$ (cf. \cite[Proposition 3.8]{Koh08}).
Since $\neg \neg \vp_1(x)\in \df{\E_{k+1}}$, 
By $\PNFT{\df{\E_{k+1}}}{\Sigma_{k+1}}$, there exists $\vp_1'(x) \in \Sigma_{k+1}(x)$ such that $\T \vdash \neg \neg \vp_1(x) \lr \vp_1' (x)$.
Since $\pa$ is an extension of $\T$ and $\pa \vdash \vp$, we have $\pa \vdash \vp_1' (x) \to \vp_1(x)$.
Then, by Lemma \ref{lem: sigma_k and pi_k} and Theorem \ref{thm: conservation results for PNF -> Pi_k+2}, we have that
$\ha + \LEM{\Sigma_{k}} $ proves $ \vp_1' (x) \to \vp_1(x)$, and hence, 
$\forall x \left(
\vp_1' (x) \to \vp_1(x)
\right)$.
Since $\T$ is an extension of $\ha  + \LEM{\Sigma_{k}}$, we have $\T \vdash \vp$.
\end{proof}

\begin{lemma}
\label{lem: PNFT(Ek,Sk) + Pk-1-LEM => PNFT(NEk, NNSk)}
Let $\T $ be an extension of $\ha + \NN{\DNE{\Sigma_{k-1}}} $ $(\ha$ if $k=0)$.
If $\PNFT{\E_{k'}}{\Sigma_{k'}}$ for all $k' \leq k$, then $\PNFT{\n{\E_{k'}}}{\dn{\Pi_{k'}}}$ for all $k' \leq k$.
\end{lemma}
\begin{proof}
Assume $\PNFT{\E_{k'}}{\Sigma_{k'}}$ for all $k' \leq k$.
%By Lemma \ref{lem: PNFT(Ekm, Sk) + Pk-1-LEM => Sk-DNE}, we have
%\begin{equation}
%\label{eq: T |- Sk-DNE}
%\T \vdash \DNE{\Sigma_k}.
%\end{equation}
Fix $k' \leq k$ and
%For each 
$\vp \in \E_{k'}$.
By $\PNFT{\E_{k'}}{\Sigma_{k'}}$, there exists $\vp'\in \Sigma_{k'}$ such that $\FV{\vp} = \FV{\vp'}$ and $\T \vdash \vp \lr \vp'$.
Then
\begin{equation*}
\label{eq: vp <-> vp'}
\T \vdash \neg \vp \lr \neg \vp'.
\end{equation*}
On the other hand, by Lemma \ref{lem: NP and NS}.\eqref{item: NS}, there exists $\vp'' \in \Pi_{k'}$ such that  $\FV{\vp'} = \FV{\vp''}$ and $\ha + \DNE{\Sigma_{k'-1}} \vdash \neg \vp' \lr \vp''$.
Then, by Corollary \ref{cor: HA+P |- vp_1 <-> vp_2 => HA+NNP |- NNvp_1 <-> NNvp_2}, we have
\begin{equation*}
\label{eq: nvp' <-> nnvp'}
\ha + \NN{\DNE{\Sigma_{k'-1}}} \vdash \neg \vp' \lr \neg \neg \vp'' .
\end{equation*}
%By %\eqref{eq: T |- Sk-DNE}, \eqref{eq: vp <-> vp'} and \eqref{eq: nvp' <-> nnvp'},
Then $\FV{\neg \vp} =\FV{\neg \neg \vp''}$ and $\T \vdash \neg \vp \lr \neg \neg \vp''$.
Thus we have shown $\PNFT{\n{\E_{k'}}}{\dn{\Sigma_{k'}}}$.
\end{proof}

\begin{theorem}
\label{thm: characterization of PNFT(Ek,Sk)}
Let $\T $ be a theory in-between $\ha + \LEM{\Pi_{k-1}}$ $(\ha$ if $k=0)$ and $\pa$.
%n extension of $\ha$ and  subtheory of $\pa$.
Then $\T \vdash \DNE{\Sigma_k} + \DNS{\U_k^+}$ if and only if $\PNFT{\E_{k'}}{\Sigma_{k'}}$ for all $k' \leq k$.
\end{theorem}
\begin{proof}
The ``only if'' direction is immediate from Theorem \ref{thm: PNFT}.\eqref{eq: item for Sigma_k in PNFT} and Lemma \ref{lem: sigma_k and pi_k}.
We show the converse direction.
Assume $\PNFT{\E_{k'}}{\Sigma_{k'}}$ for all $k' \leq k$.
Let $k>0$ without loss of generality. 
By Lemma \ref{lem: PNFT(Ekm, Sk) + Pk-1-LEM => Sk-DNE}, we have $\T \vdash \DNE{\Sigma_{k}}$.
In addition, by Lemma \ref{lem: PNFT(Ek,Sk) + Pk-1-LEM => PNFT(NEk, NNSk)} and Theorem \ref{thm: characterization of PNFT(NNUkp,NNPk)}, we have $\T \vdash  \DNS{\U_k^+}$.
\end{proof}

\begin{remark}
It is still open whether the assumption that $\T$ contains $\LEM{\Pi_{k-1}}$ can be omitted in Theorem \ref{thm: characterization of PNFT(Ek,Sk)}.
\end{remark}

\begin{remark}
\label{rem: impossibility of PNFT}
Akama et al. \cite{ABHK04} shows that $\LEM{\Pi_k}$ does not derive $\DNE{\Sigma_k} $ and $\DNE{\Sigma_k}$  does not derive $\DNE{(\Pi_k \lor \Pi_k)}$.
%the arithmetical hierarchy of the logical principles, which is summarized in \cite[Figure 2]{ABHK04}.
%According to the latter, $\LEM{\Pi_k}$ (which derives $\DNS{\U_k^+}$ by Corollary \ref{cor: pk v pk -DNE -> Ukp-DNE}) is the strongest principle which does not derive $\DNE{\Sigma_k} $, and $\DNE{\Sigma_k}$ is so for $\DNE{(\Pi_k \lor \Pi_k)}$.
Theorem \ref{thm: characterization of PNFT(Ek,Sk)} reveals that the prenex normal form theorem for $\E_k$ and $\Sigma_k$ does not hold in $\ha + \LEM{\Pi_k}$, and Theorem \ref{thm: characterization of PNFT(Ukp,Pk)} reveals that the prenex normal form theorem for $\U_k$ and $\Pi_k$ does not hold in $\ha + \DNE{\Sigma_k}$.
\end{remark}

\begin{corollary}
\label{cor: characterization of PNFT(Ek,Sk)+PNFT(Uk,Pk)}
Let $\T $ be a theory in-between $\ha $
%+ \LEM{\Pi_{k-1}}$ $(\ha$ if $k=0)$ 
and $\pa$.
%n extension of $\ha$ and  subtheory of $\pa$.
Then $\T \vdash \DNE{\Sigma_k} + \DNE{\left( \Pi_k \lor \Pi_k \right)}$ if and only if $\PNFT{\E_{k'}}{\Sigma_{k'}}$ and $\PNFT{\U_{k'}}{\Pi_{k'}}$ for all $k' \leq k$.
\end{corollary}
\begin{proof}
Let $\T $ be a theory in-between $\ha $
%+ \LEM{\Pi_{k-1}}$ $(\ha$ if $k=0)$ 
and $\pa$.
The ``only if'' direction follows from Theorem \ref{thm: PNFT}, Corollary \ref{cor: pk v pk -DNE -> Ukp-DNE} and Lemma \ref{lem: sigma_k and pi_k}.

For the converse direction, assume that $\PNFT{\E_{k'}}{\Sigma_{k'}}$ and $\PNFT{\U_{k'}}{\Pi_{k'}}$ hold for all $k' \leq k$.
By Theorems \ref{thm: characterization of PNFT(Ukp,Pk)}, we have $\T \vdash \DNE{\left( \Pi_k \lor \Pi_k \right)}$.
%Assume $\T \vdash \DNE{\Sigma_k} + \DNE{\left( \Pi_k \lor \Pi_k \right)}$.
Since $\LEM{\Pi_{k-1}}$ is derived from $ \DNE{\left( \Pi_k \lor \Pi_k \right)}$ (cf. \cite[Theorem 3.1(1)]{ABHK04}), by Theorem \ref{thm: characterization of PNFT(Ek,Sk)}, we also have $\T \vdash \DNE{\Sigma_k}$.
\end{proof}

In the following, we show the optimality of Theorem \ref{thm: PNFT for df-formulas} (see Theorem \ref{thm: Optimality of PNFT for df-formulas}).

\begin{lemma}\label{lem: PNFT(NNUkm, Pk) + Pk-2-LEM => Sk-1-DNE}
Let $\T $ be a theory in-between $\ha +\LEM{\Pi_{k-2}} $ $(\ha$ if $k<2)$ and $\pa$.
If $\PNFT{\df{\left(\dn{\U_{k'}}\right)}}{\Pi_{k'}}$ for all $k' \leq k$, then $\T \vdash \DNE{\Sigma_{k-1}}$.
\end{lemma}
\begin{proof}
By induction on $k$.
The base case is trivial.
For the induction step, assume the assertion for $k$ and let $\T$ be a theory in-between $\ha +\LEM{\Pi_{k-1}}$ and $\pa$.
Assume also that $\PNFT{\df{\left(\dn{\U_{k'}}\right)}}{\Pi_{k'}}$ holds for all $k' \leq k+1$.
Then, by induction hypothesis, $\T$ proves $\DNE{\Sigma_{k-1}}$.
Since $\T$ contains $\ha +\LEM{\Pi_{k-1}},$ we have $\T \vdash \LEM{\Sigma_{k-1}}$ by \cite[Theorem 3.1(ii)]{ABHK04}.
Fix an instance of $\DNE{\Sigma_{k}}$
$$
\vp :\equiv \forall x ( \neg \neg \vp_1(x) \to \vp_1(x)),
$$
where $\vp_1(x) \in \Sigma_{k} (x)$.
Without loss of generality, one can assume that $\vp_1(x)$ does not contain $\lor$ (cf. \cite[Proposition 3.8]{Koh08}).
From  the  perspective  of  Remark \ref{rem: identification on Sigma_k and Pi_k},  $\vp_1(x)$ is in $\Pi_{k+1} (x)$.
Then, by $\PNFT{\df{\left(\dn{\U_{k+1}}\right)}}{\Pi_{k+1}}$, there exists $\vp_1'(x)\in \Pi_{k}(x)$ such that $\T \vdash \neg \neg \vp_1(x) \lr \vp_1' (x)$.
Since $\pa$ is an extension of $\T$ and $\pa \vdash \vp $, we have $\pa \vdash  \vp_1' (x) \to \vp_1(x)$.
By Theorem \ref{thm: conservation results for PNF -> Pi_k+2}, we have that
$\ha + \LEM{\Sigma_{k-1}}$ proves  $\vp_1' (x) \to \vp_1(x)$, and hence,
$\forall x \left( \vp_1' (x) \to \vp_1(x) \right)$.
Since $\T$ is an extension of $\ha + \LEM{\Sigma_{k-1}}$, we have
$T \vdash \vp$.
%Thus we have shown $\T \vdash \DNE{\Sigma_{k}}$.
\end{proof}

\begin{theorem}
\label{thm: Optimality of PNFT for df-formulas}
\noindent
\begin{enumerate}
    \item 
        \label{item: Characterization of PNFT(E_k^df, Sigmak)}
    Let $\T $ be a theory in-between $\ha + \LEM{\Pi_{k-1}}$ $(\ha$ if $k=0)$ 
and $\pa$.
%n extension of $\ha$ and  subtheory of $\pa$.
Then $\T \vdash \DNE{\Sigma_k}$  if and only if $\PNFT{\df{\E_{k'}}}{\Sigma_{k'}}$ for all $k' \leq k$.
    \item
    \label{item: Characterization of PNFT(U_k^df, Pik)}
   Let $\T $ be a theory in-between $\ha + \LEM{\Pi_{k-2}}$ $(\ha$ if $k<2)$ 
and $\pa$.
%n extension of $\ha$ and  subtheory of $\pa$.
Then $\T \vdash \DNE{\Sigma_{k-1}}$  if and only if $\PNFT{\df{\U_{k'}}}{\Pi_{k'}}$ for all $k' \leq k$.
\end{enumerate}
\end{theorem}
\begin{proof}
\eqref{item: Characterization of PNFT(E_k^df, Sigmak)}:
The ``only if'' direction is by Theorem \ref{thm: PNFT for df-formulas}.\eqref{item: Characterization of PNFT(E_k^df, Sigmak)}.
The converse direction is by Lemma \ref{lem: PNFT(Ekm, Sk) + Pk-1-LEM => Sk-DNE}.
 
\eqref{item: Characterization of PNFT(U_k^df, Pik)}:
The ``only if'' direction is by Theorem \ref{thm: PNFT for df-formulas}.\eqref{item: Characterization of PNFT(U_k^df, Pik)}.
Note that any formula in $\df{\left(\dn{\U_{k'}}\right)}$ is in $\df{\U_{k'}}$.
Then the converse direction follows from Lemma \ref{lem: PNFT(NNUkm, Pk) + Pk-2-LEM => Sk-1-DNE}.
\end{proof}

At the end of this section, we characterize some variants of prenex normal form theorems.
%In particular, we show the optimality of the prenex normal form theorems for
%$\E_{k}^{\rm df}, \U_{k}^{\rm df} (Theorem \ref{thm: Optimality of PNFT for df-formulas}) and $\dn{\U_k}$.
%$\PNFT{\E_{k}^{\rm df}}{\Sigma_{k}}$,
%$\PNFT{\U_{k}^{\rm df}}{\Sigma_{k}}$
% and 

\begin{theorem}
\label{thm: characterization of PNFT(NEk,Pk)}
Let $\T $ be a theory in-between $\ha + \LEM{\Pi_{k-2}}$ $(\ha$ if $k<2)$ and $\pa$.
%n extension of $\ha$ and  subtheory of $\pa$.
Then $\T \vdash \DNE{\Sigma_{k-1}} +\DNS{\U_k^+}$ if and only if $\PNFT{\dn{\U_{k'}}}{\Pi_{k'}}$ for all $k' \leq k$.
\end{theorem}
\begin{proof}
We first show the ``only if'' direction.
Let $k>0$ without loss of generality.
Assume $\T \vdash \DNE{\Sigma_{k-1}} +\DNS{\U_k^+}$ and fix $k' \leq k$.
%Since $\DNE{\Pi_{k'}}$ implies $\DNS{\Pi_{k'}}$, by Theorem \ref{thm: PNFT}.\eqref{eq: item for Pi_k in PNFT} and Lemma \ref{lem: k -> k-1}.\eqref{item: sk-DNE -> pk+1-DNE}, we have $\T \vdash \DNS{\U_{k'}^+}$. 
Since $\T \vdash \DNS{\U_{k'}^+}$, by Lemma \ref{PNFT_of_NE_k_and_NNU_k}, for any $\vp \in \U_{k'}^+$, there exists $\vp'\in \Pi_{k'}$ such that $\FV{\vp} = \FV{\vp'}$ and
$$
% \ha + \DNS{\U_{k'}^+} 
\T \vdash \neg \neg \vp \lr \neg \neg \vp' .
$$
Since $\T \vdash \DNE{\Sigma_{k-1}}$, by Lemma \ref{lem: k -> k-1}.\eqref{item: sk-DNE -> pk+1-DNE}, we have $T \vdash  \neg \neg  \vp' \lr \vp'$, and hence, $T \vdash  \neg \neg  \vp \lr \vp'$.
Thus we have $\PNFT{\dn{\U_{k'}}}{\Pi_{k'}}$.

Next, we show the converse direction.
Assume that $\PNFT{\dn{\U_{k'}}}{\Pi_{k'}}$ holds for all $k' \leq k$.
By Lemma \ref{lem: PNFT(NNUkm, Pk) + Pk-2-LEM => Sk-1-DNE}, we have $\T \vdash \DNE{\Sigma_{k-1}}$.
\begin{comment}
by induction on $k$.
The base case is trivial.
For the induction step, assume the assertion for $k$ to show that for $k+1$.
Assume that $\T$ is a theory in-between $\ha + \LEM{\Pi_{k-1}}$ and $\pa$, and $\PNFT{\dn{\U_{k'}}}{\Pi_{k'}}$ holds for all $k' \leq k+1$.
We first show $\T \vdash \DNE{\Sigma_{k}}$.
By induction hypothesis, we have $\T \vdash \DNE{\Sigma_{k-1}}$.
Since $T$ contains $\LEM{\Pi_{k-1}}$, by \cite[Theorem 3.1.(ii)]{ABHK04}, we have $\T \vdash \LEM{\Sigma_{k-1}}$.
Fix an instance of $\DNE{\Sigma_k}$
$$
\vp :\equiv \forall x (\neg \neg \vp_1(x) \to \vp_1(x))
$$
where $ \vp_1(x) \in \Sigma_k(x)$.
From the perspective of Remark \ref{rem: identification on Sigma_k and Pi_k}, $ \vp_1(x)$ is in $\Pi_{k+1}(x)$.
Then, by 
%Lemma \ref{lem: sigma_k and pi_k} and 
$\PNFT{\dn{\U_{k+1}}}{\Pi_{k+1}}$, there exists $\vp_1'(x)\in \Pi_{k+1}(x)$ such that $\T \vdash \neg \neg \vp_1(x) \lr \vp_1' (x)$.
Since $\pa$ is an extension of $\T$ and $\pa \vdash \neg \neg \vp_1(x) \to \vp_1(x)$, we have $\pa \vdash  \vp_1' (x) \to \vp_1(x)$.
By Theorem \ref{thm: conservation results for PNF -> Pi_k+2}, we have
$$
\ha + \LEM{\Sigma_{k-1}} \vdash \vp_1' (x) \to \vp_1(x),
$$
and hence,
$$
\T \vdash \neg \neg \vp_1 (x) \to \vp_1(x).
$$
%by \eqref{eq: T |- NNvp_1 <-> vp_1''}.
Then $T \vdash \vp$.
Thus we have shown $\T \vdash \DNE{\Sigma_k}$.
\end{comment}
%Then, as in the proof of ``only if'' direction, 
On the other hand, by the assumption, we have $\PNFT{\dn{\U_{k'}}}{\dn{\Pi_{k'}}}$ for all $k' \leq k$, and hence,
%By Theorem \ref{thm: characterization of PNFT(NEk,Pk)}, we have also 
$\T \vdash \DNS{\U_{k}^+}$
by Theorem \ref{thm: characterization of PNFT(NNUkp,NNPk)}.
\end{proof}

%The following theorem states 
\begin{theorem}
\label{thm: characterization of PNFT(NEkdf, NNPk) and PNFT(NNUkdf, NNPk)}
%Let $\T $ be a theory in-between $\ha + \neg \neg \LEM{\Pi_{k-2}}$ $(\ha$ if $k<2)$ and $\pa$.
%Then $\T \vdash \neg \neg \DNE{\Sigma_{k-1}}$ if and only if $\PNFT{\df{\left(\n{\E_{k'}}\right)}}{ \dn{\Pi_{k'}}}$ for all $k' \leq k$.
Let $\T $ be a theory in-between $\ha + \neg \neg \LEM{\Pi_{k-2}}$ $(\ha$ if $k<2)$ and $\pa$.
The following are pairwise equivalent:
\begin{enumerate}
    \item 
    \label{item: PNFT(NEkdf, NNPk)}
    $\PNFT{\df{\left(\n{\E_{k'}}\right)}}{ \dn{\Pi_{k'}}}$ for all $k' \leq k$;
    \item
        \label{item: PNFT(NNUkdf, NNPk)}
    $\PNFT{\df{\left(\dn{\U_{k'}}\right)}}{ \dn{\Pi_{k'}}}$ for all $k' \leq k$;
  %  \item
%        \label{item: T|-NNPkPk-DNE}
 %   $\T \vdash \NN{\DNE{(\Pi_k\lor \Pi_k)}}$;
    \item
    \label{item: T|-NNSk-1DNE}
    $\T \vdash   \neg \neg \DNE{\Sigma_{k-1}}$.
\end{enumerate}
\end{theorem}
\begin{proof}
The  equivalence  of  \eqref{item: PNFT(NEkdf, NNPk)}  and  \eqref{item: PNFT(NNUkdf, NNPk)}  is  trivial  (cf.   the  proof  of  Lemma \ref{PNFT_of_NE_k_and_NNU_k}).
In addition, $(\ref{item: T|-NNSk-1DNE} \to \ref{item: PNFT(NEkdf, NNPk)})$ is immediate from the item \ref{eq: item for NE_k in PNFT for df-formulas} in the proof of Theorem \ref{thm: PNFT for df-formulas}.
Then it suffices to show $( \ref{item: PNFT(NEkdf, NNPk)} \to \ref{item: T|-NNSk-1DNE})$.
We show this by induction on $k$.
The base case is trivial.
For the induction step, assume the assertion for $k$ and let $\T$ be a theory in-between  $\ha + \neg \neg \LEM{\Pi_{k-1}}$ and $\pa$.
Assume also that $\PNFT{\df{\left(\n{\E_{k'}}\right)}}{ \dn{\Pi_{k'}}}$ holds for all $k' \leq k+1$.
Then, by induction hypothesis, we have
$\T \vdash \neg \neg \DNE{\Sigma_{k-1}}$.
Since $\T$ contains $\ha +\neg \neg 
\LEM{\Pi_{k-1}},$ we have $\T \vdash \neg \neg \LEM{\Sigma_{k-1}}$ by \cite[Theorem 3.1(ii)]{ABHK04}.
Fix an instance of $\neg \neg \DNE{\Sigma_{k}}$
$$
\vp :\equiv  \neg \neg \forall x ( \neg \neg \vp_1(x) \to \vp_1(x)),
$$
where $\vp_1(x) \in \Sigma_{k} (x)$.
Without loss of generality, one can assume that $\vp_1(x)$ does not contain $\lor$ (cf. \cite[Proposition 3.8]{Koh08}).
%By Lemma \ref{lem: basic facts on our classes}, we have
Note $\forall x ( \neg \neg \vp_1(x) \to \vp_1(x)) \in
\df{\U_{k+1}}$, and hence, $\vp \in \df{\left(\n{\E_{k+1}}\right)}$.
%by definition.
By $\PNFT{\df{\left(\n{\E_{k+1}}\right)}}{ \dn{\Pi_{k'}}}$, there exists a sentence $\vp'\in \Pi_{k+1}$ such that
%\begin{equation*}
%\label{eq: T |- vp <-> NNvp'}
$\T \vdash \vp \lr \neg \neg \vp' $.
%\end{equation*}
Since $\pa$ is an extension of $\T$ and $\pa \vdash \vp$, we have $\pa \vdash \vp'$.
By Theorem \ref{thm: conservation results for PNF -> Pi_k+2}, we have $    \ha + \LEM{\Sigma_{k-1}} \vdash \vp'$.
Then, by Lemma \ref{lem: HA+P |- vp => HA+NNP |- NNvp}, we have
\begin{equation*}
    \ha + \neg \neg \LEM{\Sigma_{k-1}} \vdash \neg \neg \vp'.
\end{equation*}
Since $\T $ is an extension of $\ha + \neg \neg \LEM{\Sigma_{k-1}}$,
%contains $\neg \neg \LEM{\Sigma_{k-1}}$,
%by \eqref{eq: T |- vp <-> NNvp'}, 
we have $\T \vdash \vp$.
%Thus we have shown $\T \vdash \neg \neg \DNE{\Sigma_{k}}$.
\end{proof}

All of our characterization results are of the following form: For any theory $\T $ in-between $\ha + {\rm Q}_k$ and $\pa$,
%n extension of $\ha$ and  subtheory of $\pa$.
$\T \vdash {\rm P}_k$ if and only if $\PNFT{\Gamma_{k'}}{\Delta_{k'}}$ holds for all $k' \leq k$, where ${\rm P}_k, {\rm Q}_k$ are logical principles and $\Gamma_{k'}, \Delta_{k'}$
are classes of formulas.
Based on this representation, our results are summarized in
Table \ref{table: PNFT}, where it is also possible to replace $\DNS{\U_k^+}$ by $\DNS{\U_k}$ (see Remark \ref{rem: UkpDNS <-> UkDNS}).
\begin{table}[htbp]
    \centering
$$
    \begin{array}{|c|c|c|c|}
\hline
{\rm P}_k& (\Gamma_k, \Delta_k)&{\rm Q}_k& \\
\hline
\hline
%\vspace*{5pt}
\NN{\DNE{\Sigma_{k-1}}}&\left( \df{\left(\dn{\U_k}\right)}, \dn{\Pi_k}  \right)&\NN{\LEM{\Pi_{k-2}}}&\text{Theorem \ref{thm: characterization of PNFT(NEkdf, NNPk) and PNFT(NNUkdf, NNPk)}}\\[4pt]
\hline
\DNS{\U_k^+} & \left( \dn{\U_k}, \dn{\Pi_k}  \right) &\emptyset & \text{Theorem \ref{thm: characterization of PNFT(NNUkp,NNPk)}}\\[2pt]
 \hline
 \DNE{\Sigma_{k-1}}& \left( \df{\U_k}, \Pi_k  \right) & \LEM{\Pi_{k-2}}&\text{Theorem \ref{thm: Optimality of PNFT for df-formulas}.\eqref{item: Characterization of PNFT(U_k^df, Pik)}}\\[2pt]
 \hline
\DNE{\Sigma_{k-1}}+\DNS{\U_k^+} &\left( \dn{\U_k}, \Pi_k  \right) & \LEM{\Pi_{k-2}}& \text{Theorem \ref{thm: characterization of PNFT(NEk,Pk)}}\\[2pt]
\hline
\DNE{\Sigma_k} & \left(\df{\E_k}, \Sigma_k \right) & \LEM{\Pi_{k-1}}&\text{Theorem \ref{thm: Optimality of PNFT for df-formulas}.\eqref{item: Characterization of PNFT(E_k^df, Sigmak)}}\\[2pt]
\hline
\DNE{\Sigma_k} +\DNS{\U_k^+} & \left(\E_k, \Sigma_k \right) & \LEM{\Pi_{k-1}}&\text{Theorem \ref{thm: characterization of PNFT(Ek,Sk)}}\\[2pt]
\hline
\DNE{(\Pi_k \lor \Pi_k)}  & \left(\U_k, \Pi_k \right) & \emptyset &\text{Theorem \ref{thm: characterization of PNFT(Ukp,Pk)}}\\[2pt]
\hline
\DNE{\Sigma_k}+\DNE{(\Pi_k \lor \Pi_k)}  & \left(\U_k, \Pi_k \right) \& \left(\E_k, \Sigma_k \right) & \emptyset &\text{Corollary \ref{cor: characterization of PNFT(Ek,Sk)+PNFT(Uk,Pk)}}\\[2pt]
\hline
    \end{array}
    $$
    \caption{Characterizations of the prenex normal form theorems}
    \label{table: PNFT}
\end{table}

\section*{Acknowledgements}
The authors thank to Ulrich Kohlenbach for his helpful comments.
The first author was supported by JSPS KAKENHI Grant Numbers JP18K13450 and JP19J01239, and the second author by JP19K14586.

\bibliographystyle{abbrv}
\bibliography{bibliography.bib}

\begin{thebibliography}{1}

\bibitem{ABHK04}
Y.~Akama, S.~Berardi, S.~Hayashi, and U.~Kohlenbach.
\newblock An arithmetical hierarchy of the law of excluded middle and related
  principles.
\newblock In {\em Proceedings of the 19th Annual {IEEE} Symposium on Logic in
  Computer Science ({LICS}'04)}, pages 192--301. 2004.

\bibitem{Fri78}
H.~Friedman.
\newblock Classically and intuitionistically provably recursive functions.
\newblock In G.~H. {M\"u}ller and D.~S. Scott, editors, {\em Higher Set
  Theory}, pages 21--27, Berlin, Heidelberg, 1978. Springer Berlin Heidelberg.

\bibitem{FK18}
M.~Fujiwara and U.~Kohlenbach.
\newblock Interrelation between weak fragments of double negation shift and
  related principles.
\newblock {\em J. Symb. Log.}, 83(3):991--1012, 2018.

\bibitem{FK20-2}
M.~Fujiwara and T.~Kurahashi.
\newblock Conservation results on semi-classical arithmetic.
\newblock {\em preprint}.

\bibitem{Koh08}
U.~Kohlenbach.
\newblock {\em Applied proof theory: proof interpretations and their use in
  mathematics}.
\newblock Springer Monographs in Mathematics. Springer-Verlag, Berlin, 2008.

\bibitem{KS14}
U.~Kohlenbach and P.~Safarik.
\newblock Fluctuations, effective learnability and metastability in analysis.
\newblock {\em Annals of Pure and Applied Logic}, 165(1):266 -- 304, 2014.
\newblock The Constructive in Logic and Applications.

\bibitem{Tro73}
A.~S. Troelstra, editor.
\newblock {\em Metamathematical investigation of intuitionistic arithmetic and
  analysis}, volume 344 of {\em Lecture Notes in Mathematics}.
\newblock Springer-Verlag, Berlin, New York, 1973.

\bibitem{ConstMathI}
A.~S. Troelstra and D.~van Dalen.
\newblock {\em Constructivism in mathematics, An introduction, {V}ol. {I}},
  volume 121 of {\em Studies in Logic and the Foundations of Mathematics}.
\newblock North Holland, Amsterdam, 1988.

\bibitem{vD13}
D.~van Dalen.
\newblock {\em Logic and Structure}.
\newblock Universitext. Springer-Verlag London, fifth edition, 2013.

\end{thebibliography}
\end{document}